\documentclass[11pt]{amsart}

\usepackage{latexsym, amssymb, amsthm,color}
\usepackage[centertags]{amsmath}
\usepackage{pictexwd}
\usepackage[normalem]{ulem}
\usepackage{comment}

\newcommand{\dnote}[2][red]{{\leavevmode\unskip\footnotesize\hspace{0.75em}\textcolor{#1}{[#2]}}}
\newcommand{\todo}[1]{\dnote{\textbf{todo:} #1}}

\definecolor{GREEN}{rgb}{0,1,0}
\definecolor{green4}{rgb}{.1,.5,.1}
\definecolor{blue}{rgb}{0,0,1}
\definecolor{gray}{rgb}{0.5,0.5,0.5}
\definecolor{violet}{rgb}{0.7,0,1}

 \newcommand{\ep}{\end{proof}}

 \newif\ifpctex

  \newtheorem{theorem}{Theorem}
  \newtheorem{definition}{Definition}[section]

  \newtheorem{cond}[definition]{Condition}
  \newtheorem{proposition}[definition]{Proposition}
  \newtheorem{lemma}[definition]{Lemma}

  \newtheorem{corollary}[definition]{Corollary}
  \newcommand{\beCond}[2]{\Rand{\vspace{0,6cm}\tt #1}\begin{cond}[#2]
  \label{#1}} \theoremstyle{definition}
  \newtheorem{remark}[definition]{Remark}
  
  \numberwithin{equation}{section}
  \newtheoremstyle{step}{3pt}{0pt}{\itshape}{}{\bf}{}{.5em}{}

\theoremstyle{step} \newtheorem{step}{Step}

\newcommand{\E}{\mathbb{E}}

\newcommand{\R}{\mathbb{R}}

\newcommand{\N}{\mathbb{N}}

\newcommand{\CL}{\mathcal{L}}

\newcommand{\CM}{\mathcal{M}}

\newcommand{\CN}{\mathcal{N}}

\newcommand{\Rand}[1]{\marginpar{#1}} 
\newcommand{\be}[1]{\begin{equation}\label{#1}}
\newcommand{\ee}{\end{equation}}
\newcommand{\bew}[1]{\Rand{\vspace{0,6cm}\tt #1}\begin{equation*}\label{#1}}
\newcommand{\eew}{\end{equation*}}
\newcommand{\bea}[1]{\Rand{\vspace{0,6cm}\tt #1}\begin{eqnarray*}\label{#1}}
\newcommand{\eea}[1]{\end{eqnarray*}}

\newcommand{\beL}[2]{\Rand{\vspace{0,6cm}\tt #1}\begin{lemma}[#2]\label{#1}}
\newcommand{\beD}[2]{\Rand{\vspace{0,6cm}\tt #1}\begin{definition}[#2]\label{#1}}
\newcommand{\beT}[2]{\Rand{\vspace{0,6cm}\tt #1}\begin{theorem}[#2]\label{#1}}
\newcommand{\beP}[2]{\Rand{\vspace{0,6cm}\tt #1}\begin{proposition}[#2]\label{#1}}
\newcommand{\beC}[1]{\Rand{\vspace{0,6cm}\tt #1}\begin{corollary}\label{#1}}
\newcommand{\beR}[1]{\Rand{\vspace{0,6cm}\tt #1}\begin{remark}[#1]\label{#1}}

\newcommand{\Tto}{{_{\displaystyle\Longrightarrow\atop t\to\infty}}}
\newcommand{\tto}{{_{\displaystyle\longrightarrow\atop t\to\infty}}}
\newcommand{\Tno}{{_{\displaystyle\Longrightarrow\atop n\to\infty}}}
\newcommand{\tro}{{_{\displaystyle\longrightarrow\atop r\to\infty}}}
\newcommand{\tno}{{_{\displaystyle\longrightarrow\atop n\to\infty}}}

\newcommand{\tMo}{{_{\displaystyle\longrightarrow\atop M\to\infty}}}

\newcommand{\tzetao}{{_{\displaystyle\longrightarrow\atop \zeta\to 0}}}
\newcommand{\Tzetao}{{_{\displaystyle\Longrightarrow\atop \zeta\to 0}}}

\newcommand{\1}{{\bf 1}}
\newcommand{\supp}{\mathrm{supp}}

\DeclareMathAlphabet{\mathpzc}{OT1}{pzc}{m}{it}

\newcommand{\anita}[1]{{\color{blue}#1}}
\newcommand{\luis}[1]{{\color{violet}#1}}

\begin{document}

\title[Virus population under cell division]
{{\large Two level branching model for virus population under cell division}}

\author{Luis Osorio}
\address{Luis Osorio\\ Fakult\"at f\"ur Mathematik\\
Universit\"at Duisburg-Essen, Campus Essen\\  Universit\"atsstra{\ss}e~2\\
 45132 Essen \\ Germany}
 \email{luis.osorio@uni-due.de}

\author{Anita Winter}
\address{Anita Winter \\ Fakult\"at f\"ur Mathematik\\
Universit\"at Duisburg-Essen, Campus Essen\\  Universit\"atsstra{\ss}e~2\\
 45132 Essen \\ Germany}
\email{anita.winter@uni-due.de}

\thanks{This research was supported by the DFG through the SPP Priority Programme 1590 and CONACYT.}

\thispagestyle{empty}
\date{\today\\ Yule{\_}final.tex}


\subjclass[]{}

  \begin{abstract}
    \sloppy In this paper we study a two-level branching model for virus populations under cell division. We assume that the cells are carrying  virus populations which evolve as a branching particle system with competition, while the cells split according to a Yule process thereby dividing their virus populations into two independently evolving sub-populations.
    We then assume that sizes of the virus populations are huge and characterize the fast branching rate and huge population density limit as the solution of a well-posed martingale problem. While verifying tightness is quite standard, we provide a Feynman-Kac duality relation to conclude uniqueness. Moreover, the duality relation allows for a further study of the long term behavior of the model.
    \end{abstract}

 \maketitle


 \section{Introduction and motivation}
 \label{S:introduction}
 In this paper, we model the two-level population dynamics resulting from
 a system of virus populations evolving in cells, which themselves undergo splitting. The virus population dynamics is modeled as a branching particle system with competition. In the limit of large virus populations this results into continuous state branching with a logistic drift.
 Cells split as a Yule process, i.e., they branch binary but are not affected from death. When a cell splits, it divides the virus particles in a random way.

 Our model generalizes the two-level branching models first developed and studied in \cite{Wu1991,DawsonHochberg1991,Wu1992,Wu1993,Etheridge1993,Wu1994}. These multilevel structures have been used, for example, to
describe the replication, updating and transfer of digitized data sets as they pass
through information networks. Moreover, modeling with multilevel structures
found also applications in population biology. For example, the
genetic decomposition of a structured population of individuals living in colonies may be altered through births and deaths within a
given colony, while the colony itself may disappear or replicate at random times. Moreover, such ``division-within-division'' dynamics show also  in gene amplification in cancer cells and elsewhere, plasmid dynamics in bacteria, and proliferation of viral particles in host cells (\cite{Kimmel1997}).

In all the above papers technique from Laplace transforms leading to a deterministic duality relation were used by the authors. In the presence of interaction between individuals (for example, due to the competition) or between cells, these techniques are not available anymore. The first  model which considers such interaction in two-level branching model is \cite{MeleardRoelly2011}.
Here a multi-type model is considered in which the trait type of the virus can also determine the dynamics of the hosts. Unfortunately, the study of the long time behavior of these processes is very hard  and was only studied in simple cases to highlight the difficulties. In \cite{Dawson2018} a class of probability measure-valued processes which model multilevel multi-type populations undergoing resampling,
mutation, selection, genetic drift and spatial migration is analysed.
For that a generalization of the function-valued and
set-valued dual relation introduced in \cite{DawsonGreven2014} for the one-level model was provided.

In this paper we will restrict ourselves to mono-type populations but allow for a change in the population sizes due to branching.
Our interaction between particles will be due to competition among particles in the same host. We will provide a Feynman-Kac dual relation which generalizes
the exponential dual relation developed in \cite{HutzenthalerWakolbinger2007,GrevenSturmWinterZaehle}
for systems of logistic branching diffusions and  in \cite{Foucart2019} for more general continuous state branching processes with a logistic drift. For our extension from the one-level to the two-level case we will adapt ideas from \cite{HermannPfaffelhuber2020}.

This equation is inspired from \cite{Bansaye2008,BansayeTran2011} where criteria of survival and extinction are given for virus or parasite populations growing inside cells. Again our paper contributes by in contrast to these papers here allowing for interaction between the particles due to competition. After stating our duality relation, the first basic long term behavior is established.

Recently, in \cite{Meizis2019} a state space for evolving genealogies of two-level branching population has been introduced in detail. The present
paper aims to be the basis to use this state space for studying evolving genealogies of virus and parasites populations in within splitting cells in forthcoming work.\smallskip

\noindent {\bf Outline. } The rest of the paper is organized as follows: In Section~\ref{S:model} we introduce our model as a measure-valued Markov process.
We start with the individual based model, claim the tightness after suitable rescaling and derive a well-posed martingale problem which characterizes the limit. In Section~\ref{S:existence} we prove the existence of a solution to this martingale problem by verifying the tightness. In Section~\ref{S:ConvGen} we prove the uniform convergence of the generators from which we can conclude the martingale problem that any limit point satisfies. In Section~\ref{S:uniqueness} we show uniqueness of a solution by establishing the duality relation. In Section~\ref{S:longterm} we apply this duality relation to obtain first results on the longterm behavior.

\section{Introduction to the model and main results}
\label{S:model}
In this section we introduce our models and state the main results. In Subsection~\ref{Sub:particle} we start with the individual based model. In Subsection~\ref{Sub:particlef} we present the scaling regime under which the tightness is stated. In Subsection~\ref{Sub:martingale} we characterize the limit as the solution to a well-posed martingale problem and state convergence. In Subsection~\ref{Sub:longterm} we discuss the basic longterm behaviour.

As usual, given a Polish space $(X,{\mathcal O})$
and a Banach space $Y$ we denote by ${\mathcal B}(X;Y)$, ${\mathcal C}(X;Y)$ and ${\mathcal D}_Y(X;Y)$ the spaces of functions from $X$ to $Y$ that are Borel-measurable, continuous and c\`{a}dl\`{a}g, respectively. We abbreviate ${\mathcal B}(X)={\mathcal B}(X;\mathbb{R})$,
${\mathcal C}(X)={\mathcal C}(X;\mathbb{R})$ and ${\mathcal D}(X)={\mathcal D}(X;\mathbb{R})$. Moreover, we write ${\mathcal B}_b(X)$ and ${\mathcal C}_b(X)$ for the subspaces of bounded functions.

Furthermore, we denote by
${\mathcal M}_1(X)$ and ${\mathcal M}_f(X)$ the spaces of probability measures
and finite measures on $X$, respectively, defined on
the Borel-$\sigma$-algebra of $X$. For $f\in{\mathcal B}(X)$ and $\nu\in{\mathcal M}_f(X)$ with $\int |f|\,\mathrm{d}\nu<\infty$, we abbreviate
\begin{equation}
\label{e:008}
   \langle f,\nu\rangle:=\int f\,\mathrm{d}\nu.
\end{equation}
We write $\Longrightarrow$ for weak
convergence of a sequence $(\nu_n)_{n\in\mathbb{N}}$ to $\nu$ in $\mathcal M_1(X)$
respectively ${\mathcal M}_f(X)$, i.e., $\nu_n\Tno\nu$ if and only if $\langle f,\nu_n\rangle\tno \langle f,\nu\rangle$ for all bounded $f\in{\mathcal C}(X)$.

\subsection{The two-level branching particle model}
\label{Sub:particle}
In this subsection we introduce our individual based model in detail.


Denote by
${\mathcal N}_f\big(X\big)$ the subspace of ${\mathcal M}_f(X)$ of point measures.
Fix an {\em individual mass constant} $\zeta>0$.
The {\em virus population model with competition and cell division} is a Markov process, $Z^{\zeta}=(Z^\zeta_t)_{t\ge 0}$,  that takes values in
 the space
 \begin{equation}
 \label{m:001}
    {\mathcal N}_f\big({\zeta}\mathbb{N}_0\big):=\big\{\nu\in{\mathcal N}_f(\mathbb{R}):\,\nu=\sum_{i=1}^m\delta_{z_i},\,m\in\mathbb{N}_0,\zeta^{-1} z_1,...,\zeta^{-1} z_N\in\mathbb{N}_0\big\}
 \end{equation}
 and has the following dynamics: given its current state $\nu=\sum_{i=1}^{m}\delta_{z_i}$, for some $m\in\mathbb{N}$ and $z_1,...,z_m\in\zeta\mathbb{N}_0$,
\begin{itemize}
 \item {\bf Cell division. } For all $i=1,...,m$ and for each $k\in \{1,...,\zeta^{-1}z_i\}$ at rate
 \begin{equation}
 \label{e:006}
    r {\zeta^{-1}z_i\choose k} \theta^{\zeta^{-1}z_i-k}\big(1-\theta\big)^k
 \end{equation}
 we have the following jump:
 \begin{equation}
 \label{e:002}
    \nu\mapsto\nu-\delta_{z_i}+\delta_{\zeta k}+\delta_{z_i-\zeta k}.
\end{equation}
That is, for each $i=1,...,m$ at rate $r>0$ the $i^{\mathrm{th}}$ cell splits into two and each of the $\zeta^{-1}z_i$ many virus particles currently living within the $i^{\mathrm{th}}$ cell goes with probability $\theta\in(0,1)$ to one of the two new cells and with probability $(1-\theta)$ to the other. All virus particles decide independently of all others to which new cell they will belong.
\item {\bf Branching with competition within a cell. }
\begin{itemize}
\item {\bf Birth. } For all $i=1,2...,m$, each of the $\zeta^{-1} z_i$  particles  in the $i^{\mathrm{th}}$ cell  gives birth at rate $\zeta^{-1}\sigma+ K$, for $\sigma>0$ and $K\in\mathbb{R}$, to a new particle of mass $\zeta$. The new virus particle also lives in the   $i^{\mathrm{th}}$ cell. This leads to the following jump:
\begin{equation}
 \label{e:003}
    \nu\mapsto\nu-\delta_{z_i}+\delta_{z_i+\zeta}.
\end{equation}
\item {\bf Natural death. } For all $i=1,2...,m$, each of the $\zeta^{-1} z_i$ virus particles  in the $i^{\mathrm{th}}$ cell dies
 at rate $\zeta^{-1}\sigma$. This leads to the following jump:
 \begin{equation}
 \label{e:004}
    \nu\mapsto\nu-\delta_{z_i}+\delta_{z_i-\zeta}.
\end{equation}
\item {\bf Death due to competition. } For all $i=1,2...,m$, each of the $\zeta^{-1} z_i$  virus particles  in the $i^{\mathrm{th}}$ cell dies at rate $\lambda (z_i-\zeta)$ due to its lost in competition against one of the $\zeta^{-1}z_i-1$ other virus particles in the $i^{\mathrm{th}}$ cell. This leads to the following jump:
\begin{equation}
 \label{e:005}
    \nu\mapsto\nu-\delta_{z_i}+\delta_{z_i-\zeta}.
\end{equation}
\end{itemize}
 \end{itemize}

 \begin{remark}[state dependent splitting] One can generalize the above model by replacing the above splitting rate $r$
by a splitting rate which depends on the virus mass carried by the cell. In Corollary~\ref{Cor:Divxp} we consider the particular case where
\begin{equation}
   r(z)=\bar r(1+z^p),
\end{equation}
for some $p\geq 1$, $\bar r>0$.
 \hfill$\qed$
 \label{Rem:002}
 \end{remark}

To be more precise,  let $N(\mathrm{d}s,\mathrm{d}i,\mathrm{d}x)$ be the Poisson point process on
$\R_+\times \N\times \R_+$ with the intensity measure $\lambda\otimes\eta\otimes\lambda$, where $\lambda$ denotes the Lebesgue measure and $\eta$ the counting measure.
Consider a stochastic process $Z^\zeta=(Z^\zeta_t)_{t\ge 0}$ taking values in $\mathcal{N}_f(\zeta\N)$, i.e., $Z^\zeta_t$ is for all $t\ge 0$ of the form
$Z^\zeta_t=\sum_{i=1}^{\langle 1,\nu_t\rangle} \delta_{z_{i,t}}$, which solves the following stochastic differential equation:
\begin{equation}
\label{Eq:DiscrModel}
    \begin{aligned}
    Z^\zeta_t &= Z^\zeta_0
+\int_{[0,t]\times\mathbb{N}\times[0,r]}\sum_{k=1}^{\zeta^{-1}z_{i,s-}}
    \big(\delta_{\zeta k}+\delta_{z_{i,s-}-\zeta k}-\delta_{z_{i,s-}}\big)
    \\
    &\hspace{5cm}\1_{\{1,...,\langle 1,Z^\zeta_{s-}\rangle\}}(i)\1_{(rF^{\theta}_{\zeta^{-1}z_{i,s-}}(k-1),rF^{\theta}_{\zeta^{-1}z_{i,s-}}(k)]}(x)
    \,N(\mathrm{d}s,\mathrm{d}i,\mathrm{d}x)
  \\
    &+
    \int_{[0,t]\times\mathbb{N}\times\mathbb{R}_+} \big(\delta_{z_{i,s-}+\zeta}-\delta_{z_{i,s-}}\big)\1_{\{1,...,\langle 1,Z^\zeta_{s-}\rangle\}}(i)\1_{(r,r+(\zeta^{-1}\sigma+K)\zeta^{-1}z_{i,s-}]}(x)\,
    N(\mathrm{d}s,\mathrm{d}i,\mathrm{d}x)
    \\
    &+\int_{[0,t]\times\mathbb{N}
    \times\mathbb{R}_+} \big(\delta_{z_{i,s-}-\zeta}-\delta_{z_{i,s-}}\big)
    \\
    &
    \hspace{2cm}\1_{\{1,...,\langle 1,Z^\zeta_{s-}\rangle\}}(i)\1_{(r+(\zeta^{-1}\sigma+K)\zeta^{-1}z_{i,s-},r+(2\zeta^{-1}\sigma+K+\lambda(z_{i,s-}-\zeta))
    \zeta^{-1}z_{i,s-}]}(x)\, N(\mathrm{d}s,\mathrm{d}i,\mathrm{d}x)
    \end{aligned}
\end{equation}
where for $n\in\mathbb{N}_0$ and $k\in\{0,...,n\}$,
\begin{equation}
\label{e:007}
   F^\theta_{n}(k):=\sum_{j=0}^k p^\theta_{j,n}:=\sum_{j=0}^k{n\choose j} \theta^j (1-\theta)^{n-j}.
\end{equation}

Our first result states that this stochastic differential equation has a unique and Markovian solution.
\begin{proposition}[individual based model]
    Let $Z_0\in \mathcal{N}_f(\zeta \N)$ satisfy $\E\left[\langle 1+x,Z_0\rangle\right]<\infty$.
    Then there exist a unique solution to equation (\ref{Eq:DiscrModel}).
     Each solution is a Markov process whose  generator $\Omega_{\mbox{\tiny $2$-level}}^\zeta$ acts on $\mathcal{B}_{b}(N_f(\zeta\N))$ as follows: for $\phi\in \mathcal{B}_{b}(N_f(\zeta\N)))$,
\begin{equation}
    \begin{aligned}
    \label{Eq:PartGener1}
    \Omega_{\mbox{\tiny $2$-level}}^\zeta
    :=\Omega^{(r,\theta),\zeta}_{\mbox{\tiny cell split}}+\Omega^{(\sigma,K,\lambda),\zeta}_{\mbox{\tiny virus branching}}
\end{aligned}
\end{equation}
with
\begin{equation}
    \begin{aligned}
    \label{Eq:PartGener2}
    \Omega^{(r,\theta),\zeta}_{\mbox{\tiny cell split}} \phi(\nu)
    &=
    r\int_{\mathbb{R}_+} \sum_{k=0}^{\zeta^{-1}z}p^\theta_{k,\zeta^{-1}z} \Big(\phi\big(\nu-\delta_{z}+\delta_{z-\zeta k}+\delta_{\zeta k}\big)-\phi(\nu)\Big)
    \end{aligned}
    \end{equation}
    and
    \begin{equation}
    \begin{aligned}
    \label{Eq:PartGener3}
    \Omega_{\mbox{\tiny virus branching}}^{(\sigma,K,\lambda),\zeta} \phi(\nu)
    &:=
    \Omega_{\mbox{\tiny virus birth}}^{(\sigma,K,\lambda),\zeta} \phi(\nu)+\Omega_{\mbox{\tiny virus death}}^{(\sigma,K,\lambda),\zeta} \phi(\nu)
    \\
    &:=
    \big(\zeta^{-1}\sigma+K\big)\zeta^{-1}\int_{\mathbb{R}_+}   z \Big(\phi\big(\nu-\delta_{z} + \delta_{z+\zeta} \big) -\phi(\nu) \Big)\, \nu(\mathrm{d}z)
    \\
    &\,+ \zeta^{-1} \int_{\mathbb{R}_+} \big(\zeta^{-1}\sigma +\lambda (z-\zeta)\big)z \Big(\phi\big(\nu-\delta_{z} + \delta_{z-\zeta} \big) -\phi(\nu)  \Big)\, \nu(\mathrm{d}z).
    \end{aligned}
    \end{equation}
    \label{Prop:WellDef}
\end{proposition}

\begin{proof} Notice that for each $M\in\mathbb{N}$ and all $t\le \tau_M$, where
\begin{equation}
\label{e:020}
    \tau_M :=\inf\big\{t>0:\,\int_{\mathrm{R}_+} \big(1+z\big) Z^\zeta_t(\mathrm{d}z)>M\big\}
\end{equation}
all coefficients in the above sde are bounded, and the sde has therefore a well-defined solution up to time $\tau_M$. We are going to show that $\tau_M\tMo\infty$, almost surely.

For that we proceed as in the proof of \cite[Proposition~2.6]{FournierMeleard2004}.  That is, for all $M\in\mathbb{N}$, $T\ge 0$  and $t\in[0,\tau_M\wedge T]$ and  all $\phi\in {\mathcal{C}_b}(\mathcal N_f(\zeta \N))$,
\begin{equation}
\label{Eq:phi(Z)}
    \begin{aligned}
    \phi (Z^\zeta_t)
    &=
    \phi(Z^\zeta_0)+
    \int_{[0,t]\times\mathbb{N}\times[0,r]}\sum_{k=1}^{\zeta^{-1}z_{i,s-}}
    \Big(\phi\big(\nu+\delta_{\zeta k}+\delta_{z_{i,s-}-\zeta k}-\delta_{z_{i,s-}}\big)-\phi\big(\nu\big)\Big)
    \\
    &\hspace{3cm}\cdot
    \1_{\{1,...,\langle 1,Z^\zeta_{s-}\rangle\}}(i)\1_{(rF^{\theta}_{\zeta^{-1}z_{i,s-}}(k-1),rF^{\theta}_{\zeta^{-1}z_{i,s-}}(k)]}(x)\,
    N(\mathrm{d}s,\mathrm{d}i,\mathrm{d}x)
  \\
    &+
    \int_{[0,t]\times\mathbb{N}\times\mathbb{R}_+} \Big(\phi\big(\nu+\delta_{z_{i,s-}+\zeta}-\delta_{z_{i,s-}}\big)-\phi\big(\nu\big)\Big)
    \\
    &\hspace{4cm}\cdot
    \1_{\{1,...,\langle 1,Z^\zeta_{s-}\rangle\}}(i)\1_{(r,r+(\zeta^{-1}\sigma+K)\zeta^{-1}z_{i,s-}]}
    (x)\,N(\mathrm{d}s,\mathrm{d}i,\mathrm{d}x)
    \\
    &+\int_{[0,t]\times\mathbb{N}\times\mathbb{R}_+} \Big(\phi\big(\nu+\delta_{z_{i,s-}-\zeta}-\delta_{z_{i,s-}}\big)-\phi\big(\nu\big)\Big)
    \\
    &\hspace{2cm}\cdot\1_{\{1,...,\langle 1,Z^\zeta_{s-}\rangle\}}(i)\1_{(r+(\zeta^{-1}\sigma+K)\zeta^{-1}z_{i,s-},r+(2\zeta^{-1}\sigma+K+\lambda(z_{i,s-}-\zeta))\zeta^{-1}z_{i,s-}]}(x)\, N(\mathrm{d}s,\mathrm{d}i,\mathrm{d}x).
\end{aligned}
\end{equation}

In particular, for $\phi(\nu):=\int_{\mathbb{R}_+} (1+z)\,\nu(\mathrm{d}z)$ we have
\begin{equation}
\label{e:009}
    \begin{aligned}
    &\int_{\mathbb{R}_+} \big(1+z\big)\,Z^\zeta_t(\mathrm{d}z)
    \\
    &=
    \int_{\mathbb{R}_+} \big(1+z\big)\,Z^\zeta_0(\mathrm{d}z)  +
    \int_{[0,t]\times \N \times[0,r]} \1_{\{1,...,\langle 1,Z^\zeta_{s-}\rangle\}}(i)\,N(\mathrm{d}s,\mathrm{d}i,\mathrm{d}x)
    \\
    &+
    \int_{[0,t]\times\mathbb{N}\times\mathbb{R}_+} \zeta \1_{\{1,...,\langle 1,Z^\zeta_{s-}\rangle\}}(i)\1_{(r,r+(\zeta^{-1}\sigma+K)\zeta^{-1}z_{i,s-}]}(x)\,
    N(\mathrm{d}s,\mathrm{d}i,\mathrm{d}x)
    \\
    &-\int_{[0,t]\times\mathbb{N}\times\mathbb{R}_+} \zeta\1_{\{1,...,\langle 1,Z^\zeta_{s-}\rangle\}}(i)\1_{(r+(\zeta^{-1}\sigma+K)\zeta^{-1}z_{i,s-},r+(2\zeta^{-1}\sigma+K+\lambda(z_{i,s-}-\zeta))
    \zeta^{-1}z_{i,s-}]}(x)\, N(\mathrm{d}s,\mathrm{d}i,\mathrm{d}x)
    \\
    &\leq
    \int_{\mathbb{R}_+} \big(1+z\big)\,Z^\zeta_0(\mathrm{d}z)  +
    \int_{[0,t]\times \N \times[0,r]} \1_{\{1,...,\langle 1,Z^\zeta_{s-}\rangle\}}(i)\,N(\mathrm{d}s,\mathrm{d}i,\mathrm{d}x)
    \\
    &+
    \int_{[0,t]\times\mathbb{N}\times\mathbb{R}_+} \zeta \1_{\{1,...,\langle 1,Z^\zeta_{s-}\rangle\}}(i)\1_{(r,r+(\zeta^{-1}\sigma+K)\zeta^{-1}z_{i,s-}]}(x)\,
    N(\mathrm{d}s,\mathrm{d}i,\mathrm{d}x).
\end{aligned}
\end{equation}

Since the right hand side in the previous inequality is increasing in $t$, we have that
\begin{equation}
    \begin{aligned}
    \label{e:010}
    &\sup_{t\in [0,\tau_M\wedge T]}\int_{\mathbb{R}_+} \big(1+z\big)\,Z^\zeta_t(\mathrm{d}z)
    \\
    &\leq
    \int_{\mathbb{R}_+} \big(1+z\big)\,Z^\zeta_0(\mathrm{d}z)
    +\int_{[0,\tau_M\wedge T]\times \N \times[0,r]} \1_{\{1,...,\langle 1,Z^\zeta_{s-}\rangle\}}(i)\,
    N(\mathrm{d}s,\mathrm{d}i,\mathrm{d}x)
    \\
    &+
    \int_{[0,\tau_M\wedge T]\times\mathbb{N}\times\mathbb{R}_+} \zeta \1_{\{1,...,\langle 1,Z^\zeta_{s-}\rangle\}}(i)\1_{[r,r+(\zeta^{-1}\sigma+K)\zeta^{-1}z_{i,s-}]}(x)\,
    N(\mathrm{d}s,\mathrm{d}i,\mathrm{d}x).
\end{aligned}
\end{equation}

Thus, for every $\zeta \in (0,1]$,
\begin{equation}
    \begin{aligned}
    \label{e:011}
    &\E\big[\sup_{t\in [0,\tau_M\wedge T]}{\int_{\mathbb{R}_+}\big(1+z\big)}\,Z^\zeta_t(\mathrm{d}z) \big]
    \\
    &\leq
    \E\big[{\int_{\mathbb{R}_+}\big(1+z\big)}\,Z^\zeta_0(\mathrm{d}z) \big]
    +r\int_{[0,T]} \E\big[{\int_{\mathbb{R}_+}\big(1+z\big)}\,Z^\zeta_{s-}(\mathrm{d}z)  \big]\,  \mathrm{d}s
    +\int_{[0,T]} \E\big[\sum\nolimits_{i=1}^{\langle 1, Z^\zeta_{s-}\rangle}(\zeta^{-1}\sigma+K)z_{i,s-}\big]\,\mathrm{d}s
    \\
    &=
    \E\big[{\int_{\mathbb{R}_+}\big(1+z\big)}\,Z^\zeta_0(\mathrm{d}z) \big]
    +\big(r+\zeta^{-1}\sigma+K\big)\int_{[0,T]} \E\big[{\int_{\mathbb{R}_+}\big(1+z\big)}\,Z^\zeta_{s-}(\mathrm{d}z)  \big]\,\mathrm{d}s
    \\
    &\leq
    \E\big[{\int_{\mathbb{R}_+}\big(1+z\big)}\,Z^\zeta_0(\mathrm{d}z) \big]
    +\big(r+\zeta^{-1}\sigma+K\big)\int_{[0,T]} \E\big[\sup_{u\in [0,s]}
    {\int_{\mathbb{R}_+}\big(1+z\big)}\,Z^\zeta_{u-}(\mathrm{d}z) \big]\,\mathrm{d}s
\end{aligned}
\end{equation}

It follows from the Gronwall inequality that for all $m\in\mathbb{N}$,
    \begin{equation}
    \label{e:012}
    \E\big[\sup_{t\in [0,\tau_M\wedge T]}{\int_{\mathbb{R}_+}\big(1+z\big)}\,Z^\zeta_t(\mathrm{d}z) \big]
    \leq
    \E\big[{\int_{\mathbb{R}_+}\big(1+z\big)}\,Z^\zeta_0(\mathrm{d}z) \big]\cdot e^{(r+\zeta^{-1}\sigma+K)T}
    <
    \infty.
\end{equation}

This implies in particular that the sequence of stopping times $\tau_M$ tends to infinity as $m\to\infty$.
It follows  that for every fixed value $0<\zeta\leq 1$ the process $Z^\zeta=(Z^\zeta_t)_{t\geq 0}$ is well-defined. \smallskip

\noindent Also, taking expectation in (\ref{Eq:phi(Z)}) we obtain
\begin{equation}\label{Eq:Ephi}
    \begin{aligned}
    &\E\big[\phi (Z^\zeta_t)\big]-\E\big[\phi (Z^\zeta_0)\big]
    \\
    &= r\int_{0}^{t} \E\Big[\int_{\mathbb{R}_+}  \sum_{k=1}^{\zeta^{-1}z} p^\theta_{k,z_{i,s-}} \Big(\phi\big( Z^\zeta_{s-}+\delta_{k}+\delta_{z-k} -\delta_{z}\big) -\phi(Z^\zeta_{s-}) \Big) Z_s(\mathrm{d}z)\Big]\,\mathrm{d}s
    \\
    &\,
    +\big(\zeta^{-1}\sigma+K\big)\zeta^{-1}\int_{0}^{t} \E\Big[ \int_{\mathbb{R}_+} z\Big(\phi \big(Z^\zeta_{s-}+\delta_{z+\zeta}-\delta_{z}\big) - \phi(Z^\zeta_{s-})\Big)Z_s(\mathrm{d}z)\Big]\,\mathrm{d}s
    \\
    &\,
    +\zeta^{-1}z\int_{0}^{t} \E\Big[\int_{\mathbb{R}_+} \big(\zeta^{-1}\sigma+\lambda(z-\zeta)\big)\Big( \phi\big(Z^\zeta_{s-}+\delta_{z-\zeta}-\delta_{z}\big) - \phi(Z^\zeta_{s-})\Big)\Big]\, Z_s(\mathrm{d}z)\,\mathrm{d}s.
\end{aligned}
\end{equation}

Finally taking derivative in $t=0$  in equation (\ref{Eq:Ephi}) it is clear that $Z^\zeta=(Z^\zeta_t)_{t\ge 0}$ has the claimed generator.
\end{proof}

\begin{definition}[$2$-level branching particle model]
In the following we refer to $Z^\zeta$ as the $2$-level branching particle model with competition and cell division.
\label{Def:002}
\end{definition}


%

\subsection{The fast evolving two-level particle model}
\label{Sub:particlef}
In order to model the dynamics of a virus population, we want to take into consideration that particles of microorganisms have a small individual mass and replicate very fast in comparison to the splitting rate of the hosting cells. For this reason we next look for limit points in our family of population processes $\{ Z^\zeta;\,\zeta>0\}$, for which simultaneously the birth and death rates were sped up by a factor $\zeta^{-1}$, as $\zeta\to 0$.

The next tightness result will be proved in Section \ref{S:existence}.
\begin{proposition}[tightness]
    Let $\theta \in [0,1]$, $r>0$, $\sigma>0$, $\lambda \geq 0$, $K>0$ and $\{ Z^\zeta;\,\zeta>0\}$
    defined as before. Then the family of processes $\{ Z^\zeta;\,\zeta>0\}$ is tight.
    \label{P:002}
\end{proposition}

\subsection{The martingale problem}
\label{Sub:martingale}
In this subsection we present an analytic representation of the limit process in terms of a martingale problem.

We begin by introducing a class of suitable test functions, which describes the procedure of sampling $m$ different cells (if there are at least $m$ cells) and then evaluating the total number of cells together with the mass in each of the sampled cells.
\begin{definition}[polynomials] A polynomial is a function $F^{g,f,m}:\,{\mathcal N}_f(\mathbb{R}_+)\to\mathbb{R}$ of the form
\begin{equation}
\label{ASAW-b:015}
   F^{g,f,m}(\nu)
   = g\big(\langle 1,\nu\rangle\big)\int_{(\mathbb{R}_+)^m}
   f(\underline{z})\,\nu^{\otimes m,\downarrow}(\mathrm{d}\underline{z}),
\end{equation}
where $m\in\mathbb{N}$, $g\in{\mathcal C}_b(\mathbb{R}_+)$, $f\in{\mathcal C}_b(\mathbb{R}_+^m)$,
and with
\begin{equation}
\label{ASAW-b:018}
   \nu^{\otimes m,\downarrow}(\mathrm{d}\underline{z})=\nu(\mathrm{d}z_1)\big(\nu(\mathrm{d}z_2)-\delta_{z_1}(\mathrm{d}z_2)\big)... \big(\nu(\mathrm{d}z_m)-\sum_{j=1}^{m-1}\delta_{z_j}(\mathrm{d}z_m)\big).
\end{equation}
\label{Def:001}
\end{definition}

Denote by ${\mathcal D}$ the algebra generated by these monomials, and consider for each $k\in\mathbb{N}$ the subspace
\begin{equation}
\label{e:017}
   {\mathcal D}^{0,k}:=\big\{F=F^{g,f,m}:\,f\in{\mathcal C}^k(\mathbb{R}_+^m)\big\}.
\end{equation}
Notice that ${\mathcal D}$ and ${\mathcal D}^{0,k}$ are convergence determining. Too see this observe that the classes are linear combinations of test functions which evaluate samples taking in an i.d.d.-fashion and that the latter class is closed under multiplication and apply, for example, \cite[Theorem~2.7]{Loehr2013}.

Consider the operator $\Omega_{\mbox{\tiny $2$-level}}$ acting on
\begin{equation}
\label{e:DOmega}
   {\mathcal D}(\Omega_{\mbox{\tiny $2$-level}}):=\big\{F\in{\mathcal D}^{0,2}:\,\Omega_{\mbox{\tiny $2$-level}} F\mbox{ is well-defined and finite}\big\},
\end{equation}
where
\begin{equation}
\label{y:001}
\begin{aligned}
   \Omega_{\mbox{\tiny $2$-level}}
  &:=
   \Omega^{(r,\theta)}_{\mbox{\tiny cell split}}
  +
   \Omega^{(\sigma,K,\lambda)}_{\mbox{\tiny virus branching}}
\end{aligned}
\end{equation}
is reflecting
\begin{enumerate}
\item $\Omega^{(r,\theta)}_{\mbox{\tiny cell split}}$ the changes due to splitting of cells.
\item $\Omega^{(\sigma,K,\lambda)}_{\mbox{\tiny virus branching}}$ the changes of virus mass inside cells due to branching with competition of virus particles.
\end{enumerate}

We introduce the parts step by step:\smallskip

\noindent {\bf Step~1 (Changes of virus mass inside cells due to splitting) } Put
\begin{equation}
\label{y:005}
\begin{aligned}
   &\Omega^{(r,\theta)}_{\mbox{\tiny cell split}}F^{g,f,m}(\nu)
   \\
   &:=
   r \int_{\mathrm{R}_+}\Big(g\big(1+\langle 1,\nu\rangle\big) F^{1,f,m}\big(\nu-\delta_{z_0}+\delta_{\theta z_0}+\delta_{(1-\theta)z_0}\big)-g\big(\langle 1,\nu\rangle\big) F^{1,f,m}(\nu)\Big)\,\nu(\mathrm{d}z_0).
\end{aligned}
\end{equation}

\smallskip

\noindent {\bf Step~2 (Changes of virus mass inside cells due to branching with competition) } Put
\begin{equation}
\label{y:004}
\begin{aligned}
   &\Omega^{\sigma,K,\lambda}_{\mbox{\tiny virus branching}}F^{g,f,m}(\nu)
   \\
   &:=
   g\big(\langle 1,\nu\rangle\big)\int\sum_{j=1}^{m} \big(z_j(K-\lambda z_j)\tfrac{\partial}{\partial z_j}f(\underline{z})
    + \sigma z_j \tfrac{\partial^2}{\partial {z_j}^2}f(\underline{z})\big)\,\nu^{\otimes m,\downarrow}(\mathrm{d}\underline{z}).
\end{aligned}
\end{equation}\smallskip

We are now in a position to state our first main result which will be proven in the end of Section~\ref{S:uniqueness}.
\begin{theorem}[well-posed martingale problem]
Let $P\in{\mathcal M}_1({\mathcal N}_f(\mathbb{R}_+))$ with
\begin{equation}
\label{e:088}
   \int_{{\mathcal N}_f(\mathbb{R}_+)}\int_{\mathbb{R}_+} \big(1+z^{2}\big)\,Z_0(\mathrm{d}z)\,P(\mathrm{d}Z_0)<\infty.
\end{equation}
Then the
$(\Omega,{\mathcal D}(\Omega),P)$-martingale problem is well-posed.

Moreover, if $Z$ is the unique solution of the $(\Omega,{\mathcal D}(\Omega),P)$-martingale problem
and $\{Z^\zeta;\,\zeta>0\}$ is a tight family of virus population models with cell division such that ${\mathcal L}(Z_0^\zeta)\Tzetao P$, then
\begin{equation}
\label{e:089}
Z^\zeta\Tzetao Z,
\end{equation}
weakly in the Skorohod space ${\mathcal D}([0,\infty);{\mathcal N}_f(\mathbb{R}_+)$).
\label{T:001}
\end{theorem}

\begin{definition}[$2$-level branching model] We refer to the solution of the above martingale problem as the {\em $2$-level branching model with cell division and
logistic branching diffusions.}
\label{Def:006}
\end{definition}

\subsection{Results on the long-term behavior}
\label{Sub:longterm}
We conclude this section with a first result concerning the long term behaviour.

The following will be proven in Section~\ref{S:longterm}.  It states that if we start with a virus population living within a single initial cell, then at time $t$ we have $e^{rt}$ times a unit mean exponential random variable many cells, and a typical cell is free of a virus particle.
\begin{proposition}[basic long term behavior]
Let $Z$ be the $2$-level branching model with cell division and
logistic branching diffusions starting in one cell, i.e., $Z_0=\nu$ with $\langle 1,\nu\rangle=k$ for some $k\in\mathbb{N}$, and $X(k,1)$ is a Gamma distributed random walk with parameters $k$ and $1$. Then
\begin{equation}
   e^{-rt}Z_t\Tto X(k,1)\delta_0.
\end{equation}
\label{P:004}
\end{proposition}

\section{The compact containment condition}
\label{S:existence}
Let for each $\zeta>0$, $Z^\zeta$ be the {\em virus population model with competition and cell division} from Subsection~\ref{Sub:particle} (in particular, (\ref{Eq:DiscrModel})). In this section we verify the compact containment condition, i.e., for all $T\ge 0$ and $\varepsilon\in(0,1)$ there exists a compact subset $\Gamma=\Gamma_{T,\epsilon}\subset{\mathcal M}_f(\mathbb{R}_+)$ such that
\begin{equation}
\label{equ:cc}
  \inf_{\zeta>0} \mathbb{P}\big(\big\{Z_t^\zeta \not\in \Gamma_{T,\epsilon} \mbox{ for all } t \in [0,T] \big\}\big) \geq 1-\epsilon.
\end{equation}

For that we will
make use of the following moment bounds:
\begin{proposition}[uniform moment bounds]
Let $Z^\zeta:=(Z^\zeta_t)_{t\in[0,\infty)}$
be the virus population model with competition and cell division. Fix $p\in\mathbb{N}$ and $T>0$.
\begin{itemize}
\item[(i)] There exist a constant $C_p>0$ such that for all $\zeta\in(0,1]$ and $t\in[0,T]$,
\begin{equation}
\label{e:022}
   \E\big[\int_{\mathbb{R}_+}\big(1+z^p\big)\,Z_{t}^\zeta(\mathrm{d}z)\big]
  \leq
   \E\big[\int_{\mathbb{R}_+}\big(1+z^p\big)\,Z_{0}^\zeta(\mathrm{d}z)\big]\cdot  e^{C_p t}.
\end{equation}

In particular, if $\sup_{\zeta\in(0,1]}\E\big[\int_{\mathbb{R}_+}\big(1+z^p\big)\,Z_{t}^\zeta(\mathrm{d}z)\big]<\infty$, then
 \begin{equation}
 \label{e:021}
    \sup_{\zeta\in(0,1]}\sup_{t\in [0,T]}\E\big[\int_{\mathbb{R}_+}\big(1+z^p\big)\,Z_{t}^\zeta(\mathrm{d}z)\big]< \infty.
    \end{equation}
\item[(ii)]
Assume that sequence of initial conditions $\{Z_0^\zeta;\,\zeta \in [0,1]\}$ satisfies that
\begin{equation}
\label{e:018}
        \sup_{\zeta\in (0,1]}\E\Big[\int_{\mathbb{R}_+} \big(1+z^{2p}\big)\,Z_0^\zeta(\mathrm{d}z)\Big] <\infty.
\end{equation}
Then
 \begin{equation}
\label{e:019}
   \sup_{\zeta \in (0,1]} \E\Big[\sup_{t\in [0,T]}\int_{\mathbb{R}_+} \big(1+z^{p}\big)\,Z_t^\zeta(\mathrm{d}z)\Big]<\infty.
\end{equation}
\end{itemize}
\label{P:005}
\end{proposition}

\begin{proof}
{\em (i) }    Using the equality (\ref{Eq:Ephi}) we have that for any $p\geq 1$
\begin{equation}\label{Eq:xp}
    \begin{aligned}
    &\E\big[\int_{\mathbb{R}_+}\big(1+z^p\big)\,Z_{t}^\zeta(\mathrm{d}z)\big]
    \\
    &=
    \E\big[\int_{\mathbb{R}_+}\big(1+z^p\big)\,Z^\zeta_{0}(\mathrm{d}z)\big]
    + r\int_{0}^{t} \mathbb{E}\big[\int_{\mathbb{R}_+}  \sum_{k=1}^{\zeta^{-1}z} p^\theta_{k,z_{i,s-}} \big\{
    1+k^p+(\zeta^{-1}z-k)^p -(\zeta^{-1}z)^p\big\}\, Z^\zeta_s(\mathrm{d}z)\big]\,\mathrm{d}s
    \\
    &+\zeta^{-1}\big(\zeta^{-1}\sigma+K\big)\int_{0}^{t} \E\big[\int_{\mathbb{R}_+} z\big\{\big(z+\zeta\big)^p-z^p\big\}\,Z^\zeta_s(\mathrm{d}z)\big]\,\mathrm{d}s
    \\
    &+\zeta^{-2}\sigma \int_{0}^{t} \E\big[\int_{\mathbb{R}_+} z\big\{ \big(z-\zeta\big)^p-z^p\big\}\,Z^\zeta_s(\mathrm{d}z)\big]\,\mathrm{d}s
    \\
    &+\zeta^{-1}\lambda\int_{0}^{t} \E\big[\int_{\mathbb{R}_+} (z-\zeta)z\big\{ \big(z-\zeta\big)^p-z^p\big\}\,Z^\zeta_s(\mathrm{d}z)\big]\,\mathrm{d}s.
\end{aligned}
\end{equation}

     Note that $x\mapsto x^p$ is an increasing function, hence  the last term is negative and  for each $k\in \{1,2,..., \zeta^{-1}z\}$, $k^p+ (\zeta^{-1}z-k)^p-(\zeta^{-1}z)^p\leq (\zeta^{-1}z)^p$ . This implies that
\begin{equation}
    \begin{aligned}
    \label{e:023}
    &\E\big[\int_{\mathbb{R}_+}\big(1+z^p\big)\,Z_{t}^\zeta(\mathrm{d}z)\big]
    \\
    &\le
    \E\big[\int_{\mathbb{R}_+}\big(1+z^p\big)\,Z_{0}^\zeta(\mathrm{d}z)\big]
    \\
    &\;+ r\int_0^t \E\big[\int_{\mathbb{R}_+}\big(1+z^p\big)\,Z_{s}^\zeta(\mathrm{d}z)\big] \,\mathrm{d}s
    +\zeta^{-1}K\int_{0}^{t} \E\big[\int_{\mathbb{R}_+}z\big\{\big(z+\zeta\big)^p-z^p\big\}\,Z_{s}^\zeta(\mathrm{d}z)\big]\,\mathrm{d}s
    \\
    &+\zeta^{-2}\sigma \int_{0}^{t} \E\big[\int_{\mathbb{R}_+} z \big\{\big(z+\zeta\big)^p+\big(z-\zeta\big)^p-2z^p\big\}\,Z_{s}^\zeta(\mathrm{d}z)\big] \,\mathrm{d}s
\end{aligned}
\end{equation}

Notice also that $(z+\zeta)^p-z^p \leq \zeta C_{p,1}  (1+z)^{p-1} $ and $(z+\zeta)^p+(z-\zeta)^p -2x^p\leq \zeta^2 C_{p,2}  (1+z)^{p-2}$, so that
\begin{equation}
\label{Eq:Bound}
    \begin{aligned}
     &\E\big[\int_{\mathbb{R}_+}\big(1+z^p\big)\,Z_{t}^\zeta(\mathrm{d}z)\big]
    \\
    &\le
    \E\big[\int_{\mathbb{R}_+}\big(1+z^p\big)\,Z_{0}^\zeta(\mathrm{d}z)\big] + \big(r+KC_{p,1}+\sigma C_{p,2}\big)
    \int_0^t \E\big[\int_{\mathbb{R}_+}\big(1+z^p\big)\,Z_{s}^\zeta(\mathrm{d}z)\big] \,\mathrm{d}s.
\end{aligned}
\end{equation}

Therefore the Gronwall inequality implies that
 \begin{equation}
    \begin{aligned}
    \label{e:024}
    \E\big[\int_{\mathbb{R}_+}\big(1+z^p\big)\,Z_{t}^\zeta(\mathrm{d}z)\big]\leq \E\big[\int_{\mathbb{R}_+}\big(1+z^p\big)\,Z_{0}^\zeta(\mathrm{d}z)\big]\cdot e^{(r+KC_{p,1}+\sigma C_{p,2})t}.
\end{aligned}
\end{equation}

It follows that
  \begin{equation}
  \label{e:025}
    \sup_{\zeta\in(0,1]} \sup_{t\in [0,T]}\E\big[\int_{\mathbb{R}_+}\big(1+z^p\big)\,Z_{t}^\zeta(\mathrm{d}z) \big]< \infty.
    \end{equation}\smallskip

\noindent {\em (ii) }   Proceeding as in \cite{BansayeTran2011}, from the expression of the generator and from Doob's representation theorem we obtain for any $h\in \mathcal{C}_b(\R_+)$  and $\zeta\in (0,1]$ the following  local martingales
    \begin{equation}
    \begin{aligned}
    \label{Expr:M}
    M^{h,\zeta}_t
    &=
    \langle h, Z^\zeta_t \rangle- \langle h, Z^\zeta_0 \rangle
    -r\int_0^t \int_{\mathbb{R}_+} \sum_{k=0}^{\zeta^{-1}z}  p^\theta_{k,\zeta^{-1}z}
    \big\{ h(\zeta k) +h(x-\zeta k) -h(x)\big\}\,Z^\zeta_s(\mathrm{d}z) \,\mathrm{d}s
    \\
    &-\big(\zeta^{-1}\sigma+K\big) \zeta^{-1} \int_0^t  \int_{\mathbb{R}_+}   z \big\{  h(z+\zeta) -h(z) \big\}\, Z^\zeta_s(\mathrm{d}z) \,\mathrm{d}s
    \\
    &- \zeta^{-1} \int_0^t  \int_{0}^{\infty}  \big(\zeta^{-1}\sigma +\lambda (z-\zeta)\big)z \big\{  h(z-\zeta) -h(z)  \big\} Z_s^\zeta(\mathrm{d}z)\,\mathrm{d}s
    \end{aligned}
    \end{equation}
with quadratic variation given by
\begin{equation}
    \begin{aligned}
    \label{Expr:[M]}
    [M^{h,\zeta}]_t
    &=  r\int_0^t \int_{\mathbb{R}_+}\sum_{k=1}^{\zeta^{-1}z} p^\theta_{k,\zeta^{-1}z} \big\{ h(\zeta k) +h(z-\zeta k) -h(z)\big\}^2 \,Z^\zeta_s(\mathrm{d}z) \,\mathrm{d}s
    \\
    &+\big(\zeta^{-1}\sigma+K\big) \zeta^{-1}\int_0^t  \int_{\mathrm{R}_+}   z \big\{h(z+\zeta) -h(z) \big\}^2\, Z^\zeta_s(\mathrm{d}z)
    \,\mathrm{d}s
    \\
    &+\zeta^{-1}\int_0^t  \int_{\mathrm{R}_+}  \big(\zeta^{-1}\sigma +\lambda (z-\zeta)\big) z \big\{ h(z-\zeta) -h(z)  \big\}^2 \, Z_s^\zeta(\mathrm{d}z)\,\mathrm{d}s.
    \end{aligned}
    \end{equation}

 In particular, for any $p\geq 1$, $\zeta \in [0,1]$  and  $\tau^p_M=\inf\{t\geq 0:\,\int_{\mathrm{R}_+} (1+z^p) Z^\zeta_t(\mathrm{d}z)>M\}$, $M\geq 1$,
 \begin{equation}
    \begin{aligned}
    \label{e:026}
     &\int_{\mathrm{R}_+}\big(1+z^p\big) \,Z^\zeta_{t\wedge \tau^p_M}(\mathrm{d}z)
     \\
     &= M^{p,\zeta}_{t\wedge \tau^p_M}
     + \int_{\mathrm{R}_+}\big(1+z^p\big) \,Z^\zeta_{0}(\mathrm{d}z)
     +r\int_0^{t\wedge \tau^p_M} \int_{\mathrm{R}_+} \sum_{k=0}^{\zeta^{-1}z} \big\{ 1+ (\zeta k)^p+ (z-\zeta k)^p -z^p\big\}\,Z^\zeta_s(\mathrm{d}z) \,\mathrm{d}s
     \\
    &+\big(\zeta^{-1}\sigma+K\big) \zeta^{-1}\int_0^{t\wedge \tau^p_M} \int_{\mathrm{R}_+}   z \big\{\big(z+\zeta\big)^p -z^p \big\} \, Z^\zeta_s(\mathrm{d}z) \,\mathrm{d}s
    \\
    &+\int_0^{t\wedge \tau^p_M}  \int_{\mathrm{R}_+}  \big(\zeta^{-1}\sigma +\lambda (z-\zeta)\big)\zeta^{-1} z \big\{  \big(z-\zeta\big)^p -z^p \big\}  Z_s^\zeta(\mathrm{d}z)\,\mathrm{d}s
    \end{aligned}
    \end{equation}
is a semi-martingale stopped at the stopping time $\tau^p_M$. Since $z\mapsto z^p$ is an increasing function, $(z-\zeta)^p -z^p$ is negative and for any $k\in \{1,2,...,\zeta x\}$, $(\zeta k)^p+ (z-\zeta k)^p -z^p\leq z^p$. Doing a similar procedure as the one we did in equation (\ref{Eq:Bound}), one can see that
    \begin{equation}
    \begin{aligned}
    \label{e:027}
        &\sup_{t\in [0,T]}\int_{\mathrm{R}_+}\big(1+z^p\big)\, Z^\zeta_{t\wedge \tau^p_M}(\mathrm{d}z)
        \\
     &\leq
        \sup_{t\in [0,T]} M^{p,\zeta}_{t\wedge \tau_M}
     +  \int_{\mathrm{R}_+}\big(1+z^p\big)\, Z^\zeta_{0}(\mathrm{d}z)
     + B_{p,T} \sup_{t\in [0,T]}
     \sup_{t\in [0,T]}\int_{\mathrm{R}_+}\big(1+z^p\big)\, Z^\zeta_{t\wedge \tau^p_M}(\mathrm{d}z)
    \end{aligned}
    \end{equation}

Observe that due to Burholder-Davis-Gundis inequality,
    \begin{equation}
    \begin{aligned}
    \label{e:028}
    \E\big[\sup_{t\in [0,T]} M^{p,\zeta}_{t\wedge \tau^p_M}\big]
    \leq
    \E\big[\sup_{t\in [0,T]} M^{p,\zeta}_t\big]
    \leq \E\big[[M^{p,\zeta}]_T^{1/2}\big]
    \leq
    1+\E\big[[M^{p,\zeta}]_T\big]
    \end{aligned}
    \end{equation}
where the last inequality is due to the fact that $z^{1/2}\leq z$ for $z\geq1$. Also the quadratic variation of $M^{p,\zeta}$ satisfies for every times $0\leq T_1\leq T_2$
    \begin{equation}
    \begin{aligned}\label{Expr:[M^p]}
    &[M^{p,\zeta}]_{T_1}-[M^{p,\zeta}]_{T_2}
    \\
    &\leq
    4r\int_{T_1}^{T_2}  \int_{\mathrm{R}_+}\big(1+z^{2p}\big)\, Z^\zeta_{s}(\mathrm{d}z) \,\mathrm{d}s
    + \big(\zeta^{-1}\sigma+K\big) \zeta^{-1} \int_{T_1}^{T_2}  \int_{\mathrm{R}_+} z C_{1,p}\zeta^2 z^{2p-2}\, Z^\zeta_s(\mathrm{d}z) \,\mathrm{d}s
    \\
    &+\zeta^{-2}\sigma \int_{T_1}^{T_2}\int_{\mathrm{R}_+} z C_{2,p} z^{2p-2}\zeta^2\,  Z_s^\zeta(\mathrm{d}z)\,\mathrm{d}s
    \\
    &+
    \zeta^{-1}\lambda\int_{T_1}^{T_2}  \int_{0}^{\infty} z^2 C_{2,p} z^{2p-2}\zeta^2\,  Z_s^\zeta(\mathrm{d}z)\,\mathrm{d}s.
    \end{aligned}
    \end{equation}

Thus
\begin{equation}
 \begin{aligned}
 \label{Expr:E[M^p]}
    \E\big[[M^{p,\zeta}]_{T}\big]
    &\leq
    C_p\int_{0}^{T} \E\big[\int_{\mathrm{R}_+} \big(1+z^{2p-1}\big)\,Z^\zeta_s(\mathrm{d}z)\big] \,\mathrm{d}s
    +\zeta D_p\int_{0}^{T}  \E\big[\big(1+z^{2p}\big)\,Z_s^\zeta(\mathrm{d}z)\big] \,\mathrm{d}s
    \\
    &\leq
    C_p T \sup_{s\in [0,T]}\E\big[\int_{\mathrm{R}_+}\big(1+z^{2p-1}\big)\,Z^\zeta_s(\mathrm{d}z)\big]
    \\
    &+
    \zeta D_p T \sup_{s\in [0,T]}\E\big[\int_{\mathrm{R}_+}\big(1+z^{2p}\big)\,Z^\zeta_s(\mathrm{d}z)\big]
    \end{aligned}
    \end{equation}

This allow us to deduce that there exit a constant $A_{p,T}$ depending on $p$ and $T$, such that
    \begin{equation}
    \label{e:029}
     \sup_{\zeta\in (0,1]}\E\big[\sup_{t\in [0,T]}\int_{\mathrm{R}_+} \big(1+z^p\big)\,Z^\zeta_{t\wedge \tau^p_M}(\mathrm{d}z)\big]
     \leq
     A_{p,T} \sup_{\zeta\in (0,1]} \sup_{s\in [0,T]}\E\big[\int_{\mathrm{R}_+}\big(1+z^{2p}\big) Z^\zeta_s(\mathrm{d}z)\big]
     < \infty.
\end{equation}

Finally, notice that as a consequence $\tau^p_M\to\infty$, almost surely, as $M\to\infty$.
\end{proof}

\begin{corollary}[mass dependent splitting rate]
     Let $\bar r>0$, $b>0$, $\sigma>0$, $K>0$ and $r:\mathbb{R}_+\to\mathbb{R}_+$ such that $r(z)\leq \bar r(1+z^p)$ for all $z\in\mathbb{R}_+$. Let $Z^\zeta_0$ be a $(\mathcal{N}_f(\zeta \N))$-valued random variable with $\E\big[\int_{\mathbb{R}_+}\big(1+z^{2p}\big)\,Z^\zeta_0(\mathrm{z})\big]<\infty$. Then there exist a unique solution to equation (\ref{Eq:DiscrModel}) (with $r$ replaces by $r(\boldsymbol{\cdot})$). This solution is a Markov process $Z^\zeta=(Z^\zeta_t)_{t\ge 0}$ with values in $\mathcal{N}_f(\zeta \N)$. Its generator acts on functions  $\phi\in \mathcal{C}^{b}_{N_f(\zeta\N)}(\R_+)$ as follows:
\begin{equation}
    \begin{aligned}
    \label{Eq:PartGener4}
    \Omega^\zeta \phi(\nu)
    &=
    \int_{\mathbb{R}_+}  r(z)\sum_{k=0}^{\zeta^{-1} z} \big(\phi(\nu-\delta_{z}+\delta_{z-\zeta k}+\delta_{\zeta k})-\phi(\nu)\big) p^\theta_{k,\zeta^{-1} z}\, \nu(\mathrm{d}z)
    \\
    &\,+  \big(\zeta^{-1}\sigma+K\big) \zeta^{-1}\int_{\mathbb{R}_+} z \big(\phi(\nu-\delta_{z} + \delta_{z+\zeta}\big) -\phi(\nu)\big)\, \nu(\mathrm{d}z)
    \\
    &\,+ \zeta^{-1} \int_{\mathbb{R}_+} \big((\zeta^{-1}\sigma +\lambda (x-\zeta)\big)z \big(\phi(\nu-\delta_{z} + \delta_{z-\zeta} ) -\phi(\nu)  \big) \nu(\mathrm{d}z).
    \end{aligned}
    \end{equation}
     \label{Cor:Divxp}
\end{corollary}

\begin{proof}
    Follows directly from the previous result.
\end{proof}

We next state compact containment condition for our two level branching process with the state dependent splitting rate.
\begin{proposition}[compact containment]
     Let $\theta>0$, $\bar r>0$, $r(x)=\bar r (1+x^p)$, $\sigma>0$, $\lambda \geq 0$, $K>0$, and $(Z^\zeta_t)_{t\geq0}$  defined as before. Then the family $\{Z^\zeta=(Z^\zeta_t)_{t\geq 0};\,\zeta>0\}$ satisfy the compact containment condition.
     \label{P:001}
\end{proposition}

\begin{proof}
Let $\zeta\in (0,1]$. We recall that sets $\mathcal{M}_{N_0}([0,a_0])$ consisting of the measures bounded by $N_0\in\mathbb{N}$ with support on $[0,a_0]$ are compact sets on $\mathcal{M}_F(\R_+)$ equipped  with the weak topology.
Thus applying the Markov inequality it is clear that
\begin{equation}
    \begin{aligned}
    \label{e:037}
    &\mathbb{P}\big(\big\{\exists\, t\in[0,T]  \mbox{ s.t. } Z_t^\zeta \notin \mathcal{M}_{N_0}([0,a_0])\big\}\big)
    \\
    &\leq
    \mathbb{P}\big(\big\{\exists\, t\in[0,T] \mbox{ s.t. } \langle 1,Z_t^\zeta \rangle\ge  N_0\big\}\big)
    +
    \mathbb{P}\big(\big\{\exists\,
    t\in[0,T], x\in \supp (Z^\zeta_t) \mbox{ s.t } x>a_0\big\}\big)
    \\
    &\leq
    \frac{1}{N_0}\E\big[\sup_{t\in [0,T]} \langle 1, Z_t^\zeta \rangle\big]+ \frac{1}{a_0}\E\big[\sup_{t\in [0,T]} \langle x, Z_t^\zeta\rangle \big].
    \end{aligned}
    \end{equation}
Thus by taking $N_0$ and $a_0$ big enough, we get that the compact containment is satisfied.
\end{proof}

Now we will use Aldous-Rebolledo criterion to prove that the family of process $\{Z^\zeta;\,\zeta \in (0,1]\}$ is tight (see  \cite{JoffeMetivier1986}). Notice that due to \cite[Theorem~3.9.1]{EthierKurtz1986}, it is enough to prove Aldous-Rebolledo criterion only for the process $(\langle h, Z^\zeta_t \rangle)_{t\ge 0}$, where $h \in {\mathcal C}^2_b(\R_+)$.  Before stating the next result we remark that equation applying Taylor's theorem for  the function $h$, (\ref{Expr:[M]}) can be re-written as
\begin{equation}
\begin{aligned}
\label{[Mh]}
    [M^{h,\zeta}]_t
    &=  \int_0^t \int_{\mathbb{R}_+} r(z)\sum_{k=0}^{\zeta^{-1}z}p^\theta_{k,\zeta^{-1} z}\big(h(\zeta k)+h(z-\zeta k)-h(z)\big)^2 \, Z^\zeta_s(\mathrm{d}z) \,\mathrm{d}s
    \\
    &\,
    +\big(\zeta^{-1}\sigma+K\big)\zeta^{-1}\int_0^t  \int_{\mathbb{R}_+}  z\big(\zeta h'(z)+o(\zeta^2) \big)^2\, Z^\zeta_s(\mathrm{d}z) \,\mathrm{d}s
    \\
    &\,
    +\zeta^{-1} \int_0^t  \int_{0}^{\infty}  \big(\zeta^{-1}\sigma +\lambda (z-\zeta)\big)z \big(- \zeta h'(z)+o(\zeta^2)\big)^2 Z_s^\zeta(\mathrm{d}z)\,\mathrm{d}s
    \end{aligned}
    \end{equation}

\begin{proposition}[tightness] If the  assumptions in Corollary~\ref{Cor:Divxp}  are satisfied, then
     the family $\{Z^\zeta;\,\zeta\in(0,1]\}$  is tight.
\label{Prop:tightness}
\end{proposition}

\begin{proof} Fix a sequence $(\zeta_n)_{n\in\mathbb{N}}$ with $\zeta_n\tno 0$.
We need to prove that the family of semimartigales $\{(\langle h, Z_t^{\zeta_n} \rangle)_{t\ge 0};\,n\in\mathbb{N}\}$ satisfies the Aldous criterion. Thanks to the Aldous-Rebolledo criterion it is enough to apply Aldous criterion to both, the predictable finite variation process $A^{h,\zeta_n}=(A^{h,\zeta_n}_t)_{t\ge 0}$ with  $A^{h,\zeta_n}_t=\langle h, Z^{\zeta_n}_t\rangle-[M^{h,\zeta_n}]_t$ and the quadratic variation process $([M^{h,\zeta_n}]_t)_{t\ge 0}$. Let $\epsilon>0$ and $\{\tau_{n};\,n\in\mathbb{N}\}$ be a family of stopping times bounded by $T>0$ and $\alpha \in [0,\delta]$. Then by Markov inequality,
\begin{equation}
    \begin{aligned}
    \label{e:038}
    \mathbb{P}\big(\big\{\vert[M^{h,\zeta_n}]_{\tau_n+\alpha} - [M^{h,\zeta_nn}]_{\tau_n}\vert>\epsilon\big\}\big)
    \leq
    \frac{1}{\epsilon}\E\big[\vert[M^{h,\zeta_n}]_{\tau_n+\alpha} - [M^{h,\zeta_n}]_{\tau_n}\vert\big],
    \end{aligned}
    \end{equation}

Applying to (\ref{[Mh]}) similar techniques as in (\ref{Expr:E[M^p]}) and using the fact that  $h\in \mathcal{C}^2_b(\R_+)$, we get
    \begin{equation}
    \begin{aligned}
    \label{e:039}
    &\E\big[\vert[M^{h,\zeta_n}]_{\tau_n+\alpha} - [M^{h,\zeta_n}]_{\tau_n}\vert\big]
    \\
    &\leq
    9\Vert h\Vert_\infty \E\big[\int_{\tau_n}^{\tau_n+\alpha}\int_{\mathbb{R}_+} r(z)\,Z^{\zeta_n}_s(\mathrm{d}z)\,\mathrm{d}s\big]
    +\sigma \Vert h' \Vert_\infty^2\E\big[\int_{\tau_n}^{\tau_n+\alpha} \int_{\mathbb{R}_+} z\,Z^{\zeta_n}_s(\mathrm{d}z)\,\mathrm{d}s\big]
    \\
    &\;+
    \sigma \Vert h' \Vert_\infty^2\E\big[\int_{\tau_n}^{\tau_n+\alpha} \int_{\mathbb{R}_+} z\,Z^{\zeta_n}_s(\mathrm{d}z) \big]+ C_{h,\alpha,\sigma,K} o\big(n^{-1}\big)
    \\
    &\leq
    9\alpha \Vert h\Vert_\infty \E\big[ \sup_{s\in [0,T]}\int_{\mathbb{R}_+} z^p\,Z^{\zeta_n}_s(\mathrm{d}z) \big]
    +2\alpha \sigma \Vert h' \Vert_\infty^2\E\big[ \sup_{s\in [0,T]}\int_{\mathbb{R}_+} z\,Z^{\zeta_n}_s(\mathrm{d}z) \big] + C_{h,\alpha,r,\sigma,K,\lambda}  o\big(n^{-1}\big)
    \end{aligned}
    \end{equation}
which implies that
\begin{equation}
    \begin{aligned}
    \label{e:040}
    \limsup_{n\to \infty}\mathbb{P}\big(\big\{\vert[M^{h,\zeta_n}]_{\tau_n+\alpha} - [M^{h,\zeta_n}]_{\tau_n}\vert>\epsilon\big\}  \big))
    \leq \frac{\delta C_{h,\alpha,r,\sigma,K,\lambda}}{\epsilon}.
    \end{aligned}
    \end{equation}

Taking $\delta$ sufficiently small we obtain the Aldous criterion for the martingale part. In a similar way we find that
    \begin{equation}
    \begin{aligned}
    \label{e:041}
    &\E\big[\vert A^{h,\zeta_n}_{\tau_n+\alpha}- A^{h,\zeta_n}_{\tau_n}\vert\big]
    \\
    &\leq
    3\Vert h\Vert_\infty \E\big[ \int_{\tau_n}^{\tau_n+\alpha} \int_{\mathbb{R}_+} r(z)\,Z^{\zeta_n}_s(\mathrm{d}z) \,\mathrm{d}s\big]
    + \Vert h'\Vert_\infty \E\big[ \int_{\tau_n}^{\tau_n+\alpha} \int_{\mathbb{R}_+}  z(K-\lambda z)\,Z^{\zeta_n}_s(\mathrm{d}z) \,\mathrm{d}s\big]
    \\
    &\;+
    \sigma\Vert h''\Vert_\infty \E\big[ \int_{\tau_n}^{\tau_n+\alpha} \int_{\mathbb{R}_+} z\,Z^{\zeta_n}_s(\mathrm{d}z)\,\mathrm{d}s\big] +C_{h,\alpha,r,\sigma,K,\lambda} \; o\big(n^{-1}\big)
    \\
    \leq& 3\alpha \Vert h\Vert_\infty \E\big[ \sup_{s\in [0,T]} \int_{\mathbb{R}_+} \big(1+z^p\big)\,Z^{\zeta_n}_s(\mathrm{d}z)\, \big]+ \alpha C_{K,\lambda}\Vert h'\Vert_\infty \E\big[ \sup_{s\in [0,T]} \int_{\mathbb{R}_+} \big(1+z^2\big)\,Z^{\zeta_n}_s(\mathrm{d}z)\, \big]
    \\
    &\,
    +\alpha\sigma\Vert h''\Vert_\infty \E\big[ \sup_{s\in [0,T]}\int_{\mathbb{R}_+} z\,Z^{\zeta_n}_s(\mathrm{d}z)\,\,\mathrm{d}s\big] +C_{h,\alpha,r,\sigma,K,\lambda}\;
    o\big(n^{-1}\big)
    \end{aligned}
    \end{equation}
from here it is easy to see, that there exist $\delta>0$ sufficiently small such that
    \begin{align}
    \limsup_{n\to \infty}\mathbb{P}\big(\big\{ \vert A^{h,\zeta_n}_{\tau_n+\alpha} - A^{h,\zeta_n}_{\tau_n}\vert>\epsilon  \big\}\big)\leq \epsilon,
    \end{align}
which concludes the proof.
\end{proof}

\section{Uniform generator convergence}
\label{S:ConvGen}
Recall the family of $2$-level branching particle model with competition and cell division, $\{Z^\zeta;\,\zeta\in(0,1]\}$, from Definition~\ref{Def:002} and their generators $\Omega_{\mbox{\tiny $2$-level}}^{\zeta}=\Omega^{(r,\theta),\zeta}_{\mbox{\tiny cell split}}+\Omega^{(\sigma,K,\lambda),\zeta}_{\mbox{\tiny virus branching}}$ from (\ref{Eq:PartGener1}) to (\ref{Eq:PartGener3}) acting on bounded continuous functions. Furthermore recall the $2$-level branching model with cell division and logistic branching mechanism from Definition~\ref{Def:006} and its generator $(\Omega_{\mbox{\tiny $2$-level}},{\mathcal D}(\Omega_{\mbox{\tiny $2$-level}}))$
from (\ref{e:DOmega}) to (\ref{y:004}).

In this section we prove the uniform convergence of the generators.
Consider for each  $(m,M_m)\in\mathbb{N}\times\mathbb{N}$ and $L\in\mathbb{N}$ the following subspace
 \begin{equation}
 \label{SetConvDet}
 \begin{aligned}
    &\CL_{m,M_m,L}
    \\
    &:=
    \big\{F^{(g,m,f)}\in{\mathcal D}^{0,3}:\,
    g(x)\cdot(x^m\vee x^{m+1})\le M_m;\,
    \\
    &\hspace{4cm}z_i\vee z_i^2\frac{\partial^k}{\partial (z_i)^k}f((z_1,...,z_m))\le L,
    \forall k=1,2,3,(z_1,...,z_m)\in\mathbb{R}^m_+\big\}
    \end{aligned}
\end{equation}
as well as
 \begin{equation}
 \label{SetConvDetb}
    \CL:=\bigcup_{m=1}^\infty\bigcup_{M_m\in\mathbb{N}}\bigcup_{L\in\mathbb{N}}\CL_{m,M_m,L}.
 \end{equation}

 \begin{remark}[particular class] We will later construct a duality relation which used the space ${\mathcal K}$ of functions
 $F=F^{q,m,\underbar x}$ of the following form
\begin{equation}
\label{e:062}
   F^{(q,m,\underbar x)}(\nu):=q^{\langle 1,\nu\rangle}\cdot \int_{\mathbb{R}^m_+} e^{-\sum_{k=1}^m x_k z_k}\,\nu^{\otimes m,\downarrow}(\mathrm{d}(z_1,...,z_m)),
\end{equation}
where $q\in(0,1)$, $m\in\mathbb{N}$ and $(x_1,...,x_m)\in\mathbb{R}_+^m$. Notice that ${\mathcal K}\subset{\mathcal L}$.
\label{Rem:003}
 \hfill$\qed$
 \end{remark}

\begin{proposition}[convergence determining]
The set of functions ${\mathcal K}$ and $\CL$ are convergence determining in $\CM_1 (\CN_f(\R_+))$.
\label{P:003}
\end{proposition}

\begin{proof} Notice that the sets of functions $\CL$ and $\mathcal K$ defined in (\ref{SetConvDetb}) respectively in Remark~\ref{Rem:003} are  dense sets of $\mathcal{C}(\CN_f(\R_+))$. Moreover,
the algebras generated by the functions in $\CL$ respectively ${\mathcal K}$ are closed under multiplication and separate points from $\CM_1 (\CN_f(\R_+))$. Therefore $\CL$ and ${\mathcal K}$ are convergence determining (use, for example, \cite[Theorem~2.7]{Loehr2013} or \cite[Theorem~1.23]{Meizis2019}).
\end{proof}

The main result reads as follows:
\begin{proposition}[convergence of generators] For all $F\in{\mathcal L}$,
\begin{equation}
\label{e:042}
   \lim_{\zeta\to 0}\sup_{\nu\in{\mathcal N}_f(\zeta\mathbb{N}_f)}\big|\Omega^{\zeta}_{\mbox{\tiny $2$-level}}F(\nu)-\Omega_{\mbox{\tiny $2$-level}} F(\nu)\big|=0.
\end{equation}
\label{P:generators}
\end{proposition}\smallskip

To prepare the proof of  Proposition~\ref{P:generators} we first show the following lemma.
\begin{lemma}[preparatory calculations] Fix $m\in\mathbb{N}$ and $\nu\in{\mathcal N}_f(\mathbb{R}_+)$.
\begin{itemize}
\item[(i)] For $\nu$-almost all $z_0\in\mathbb{R}_+$  the following holds:
\begin{equation}
\begin{aligned}
   &
   \nu^{\otimes m,\downarrow} (\mathrm{d}(z_1,...,z_m))
   -\big(\nu-\delta_{z_0}\big)^{\otimes m,\downarrow} (\mathrm{d}(z_1,...,z_m))
   \\
   &\;
   =\sum_{i=1}^{m} 
   \big(\nu-\delta_{z_0}\big)^{\otimes (m-1),\downarrow} (\mathrm{d}(z_1,...,z_{i-1},z_{i+1},...,z_m))\otimes\delta_{z_0}(z_i).
\end{aligned}
\end{equation}
\item[(ii)] For all $z_0\in\mathbb{R}_+$  the following holds:
\begin{equation}
\begin{aligned}
   &
   \big(\nu+\delta_{z_0}\big)^{\otimes m,\downarrow} (\mathrm{d}(z_1,...,z_m))-\nu^{\otimes m,\downarrow} (\mathrm{d}(z_1,...,z_m))
   \\
   &=
   \sum_{i=1}^{m} 
   \nu^{\otimes (m-1),\downarrow} (\mathrm{d}(z_1,...,z_{i-1},z_{i+1},...,z_m))\otimes\delta_{z_0}(z_i).
\end{aligned}
\end{equation}
\item[(iii)] For  all $(w,z)\in\mathbb{R}^2_+$ the following holds:\label{MR1}
\begin{equation}
    \begin{aligned}
    \label{e:045}
    &\Big(\big(\nu+\delta_{w}+\delta_{z}\big)^{\otimes m, \downarrow}-\nu^{\otimes m, \downarrow}\Big)(\mathrm{d}(z_1,...,z_m))
    \\
    &=\sum_{i=1}^{m} 
    \nu^{\otimes (m-1),\downarrow}(\mathrm{d}(z_1,...,z_{i-1},z_{i+1},...,z_m))\otimes\delta_{w}(\mathrm{d}z_i)
    \\
    &\;
    +\sum_{i=1}^{m} 
    \nu^{\otimes (m-1),\downarrow}(\mathrm{d}(z_1,...,z_{i-1},z_{i+1},...,z_m))\otimes\delta_{z}(\mathrm{d}z_i)
    \\
    &\;
    +\sum_{i\neq j \in \{1,...,m\}} 
    \nu^{\otimes (m-2),\downarrow}(\mathrm{d}(z_1,...,z_{i\wedge j-1},z_{i\wedge j+1},...,z_{i\vee j-1},z_{i\vee j+1},...,z_m))\otimes
    \\
    &\;
    \hspace{10cm}\otimes\delta_{w}(\mathrm{d}z_i)\otimes\delta_{z}(\mathrm{d}z_j).
\end{aligned}
\end{equation}
\item[(iv)] For all $x\in\mathrm{supp}(\nu)$ and $w,z\in\mathbb{R}_+$,
\begin{equation}
\begin{aligned}
\label{Eq:nu3delta}
 &\Big(\big(\nu-\delta_{x}+\delta_{w}+\delta_z\big)^{\otimes m, \downarrow}-\nu^{\otimes m, \downarrow}\Big)(\mathrm{d}(z_1,...,z_m))
    \\
    &=
    \sum_{i=1}^{m}\big(\nu-\delta_{x}\big)^{\otimes (m-1), \downarrow}(\mathrm{d}(z_1,...,z_{i-1},z_{i+1},...,z_m))\otimes\big(-\delta_x+\delta_{w}+\delta_z\big)(\mathrm{d}z_i)
   \\
    &\;+
    \sum_{i\neq j \in \{1,...,m\}} \big(\nu-\delta_{x}\big)^{\otimes (m-2),\downarrow}(\mathrm{d}(z_1,...,z_{i\wedge j-1},z_{i\wedge j+1},...,z_{i\vee j-1},z_{i\vee j+1},...,z_m))\otimes
    \\
    &\hspace{8cm}\otimes\delta_{w}(\mathrm{d}z_i)\otimes\delta_{z}(\mathrm{d}z_j).
    \end{aligned}
    \end{equation}
\end{itemize}
\end{lemma}

\begin{proof} {\em (i) }The statement is intuitively clear as picking $m$ cells without repetition from the whole cell population can be decomposed in picking $m$ cell without repetition from all but a specific cell and in picking $(m-1)$ without repetition from all but the specific one and then adding the specific one to the sample.  To prove it formally,
we proceed by induction over $m\in\mathbb{N}$. Notice that the statement is trivial for $m=0$ and $m=1$.
If $m=2$, then
\begin{equation}
\begin{aligned}
&
\big(\nu-\delta_{z_0}+\delta_{z_0}\big)^{\otimes 2,\downarrow}(\mathrm{d}(z_1,z_2))
\\
&=
\big(\nu-\delta_{z_0}\big)^{\otimes 2,\downarrow}(\mathrm{d}(z_1,z_2))
+
\big(\nu-\delta_{z_0}\big)(\mathrm{d}(z_1))\otimes\delta_{z_0}(z_2)
\\
&\;
+
\delta_{z_0}(z_1)\otimes\big(\nu-\delta_{z_0}-\delta_{z_1}\big)(\mathrm{d}z_2)
+ 
\delta_{z_0}\otimes\delta_{z_0}(\mathrm{d}(z_1,z_2))
\\
&=
\big(\nu-\delta_{z_0}\big)^{\otimes 2,\downarrow}(\mathrm{d}(z_1,z_2))
+
\big(\nu-\delta_{z_0}\big)(\mathrm{d}(z_1))\otimes\delta_{z_0}(z_2)
+
\delta_{z_0}(z_1)\otimes\big(\nu-\delta_{z_0}\big)(\mathrm{d}z_2),
\end{aligned}
\end{equation}
which is the claim for $m=2$. Suppose next that the result is true for $m_0-1$ for some $m_0\in\mathbb{N}$. Then by the induction hypothesis,
\begin{equation}
\begin{aligned}
\label{e:044}
  &
  \nu^{\otimes m_0,\downarrow}(\mathrm{d}(z_1,...,z_{m_0}))
  \\
  &= 
  \nu^{\otimes (m_0-1),\downarrow}(\mathrm{d}(z_1,...,z_{m_0-1}))\otimes\big(\nu-\sum_{i=1}^{m_0-1}\delta_{z_i}\big)(\mathrm{d}z_{m_0})
  \\
  &=
  \big(\nu-\delta_{z_0}\big)^{\otimes (m_0-1),\downarrow}(\mathrm{d}(z_1,...,z_{m_0-1}))\otimes\big(\nu-\sum_{i=1}^{m_0-1}\delta_{z_i}\big)(\mathrm{d}z_{m_0})
  \\
  &\;
  +\sum_{i=1}^{m_0-1}
  \big(\nu-\delta_{z_0}\big)^{\otimes (m_0-2),\downarrow}(\mathrm{d}(z_1,...,z_{i-1},z_{i+1},...,z_{m_0-1}))\otimes\delta_{z_0}(z_i)
  \otimes\big(\nu-\sum_{i=1}^{m_0-1}\delta_{z_i}\big)(\mathrm{d}z_{m_0})
  \\
  &=
  \big(\nu-\delta_{z_0}\big)^{\otimes m_0),\downarrow}(\mathrm{d}(z_1,...,z_{m_0}))
  +
  \nu^{\otimes (m_0-1),\downarrow}(\mathrm{d}(z_1,...,z_{m_0-1}))\otimes\delta_{z_0}(\mathrm{d}z_{m_0})
  \\
  &\;
  +
  \big(\nu-\delta_{z_0}\big)^{\otimes (m_0-2),\downarrow}(\mathrm{d}(z_1,...,z_{i-1},z_{i+1},...,z_{m_0-1}))\otimes\delta_{z_0}(z_i)
  \otimes\big(\nu-\sum_{i=1}^{m_0-1}\delta_{z_i}\big)(\mathrm{d}z_{m_0})
  \\
&=
\big(\nu-\delta_{z_0}\big)^{\otimes m_0),\downarrow}(\mathrm{d}(z_1,...,z_{m_0}))
  +\sum_{i=1}^{m_0}
    \big(\nu-\delta_{z_0}\big)^{\otimes (m_0-1),\downarrow}(\mathrm{d}(z_1,...,z_{i-1},z_{i+1},...,z_{m_0-1}))\otimes\delta_{z_0}(z_i).
\end{aligned}
\end{equation}
This completes the proof of our proposition.\smallskip

\noindent  {\em (ii) } follows by an analogous line of argument as in (i).\smallskip

\noindent  {\em (iii) } follows by iterative application of (i).\smallskip

\noindent {\em (iv) }Notice that
\begin{equation}
    \begin{aligned}
    \label{e:049}
    &\Big(\big(\nu-\delta_{x}+\delta_{w}+\delta_z\big)^{\otimes m, \downarrow}-\nu^{\otimes m, \downarrow}\Big)(\mathrm{d}(z_1,...,z_m))
    \\
    &=
    \Big(\big(\nu-\delta_{x}+\delta_{w}+\delta_z\big)^{\otimes m, \downarrow}-\big(\nu-\delta_x\big)^{\otimes m, \downarrow}\Big)(\mathrm{d}(z_1,...,z_m))
     \\
    &\;
    +\Big(\big(\nu-\delta_{x}\big)^{\otimes m, \downarrow}-\nu^{\otimes m, \downarrow}\Big)(\mathrm{d}(z_1,...,z_m))
\end{aligned}
\end{equation}

From Lemma~\ref{MR1} we therefore know that for all $m\in\mathbb{N}$, $w,z\in \R_+$ and $x\in \supp (\nu)$,
\begin{equation}
    \begin{aligned}
    \label{e:047}
    &\Big(\big(\nu-\delta_{x}+\delta_{w}+\delta_z\big)^{\otimes m, \downarrow}-\nu^{\otimes m, \downarrow}\Big)(\mathrm{d}(z_1,...,z_m))
    \\
    &=
    \sum_{i=1}^{m}\big(\nu-\delta_{x}\big)^{\otimes (m-1), \downarrow}(\mathrm{d}(z_1,...,z_{i-1},z_{i+1},...,z_m))\otimes\delta_{w}(\mathrm{d}z_i)
    \\
    &\;
    +\sum_{i=1}^{m}\big(\nu-\delta_{x}\big)^{\otimes (m-1), \downarrow}(\mathrm{d}(z_1,...,z_{i-1},z_{i+1},...,z_m))\otimes\delta_{z}(\mathrm{d}z_i)
    \\
    &\;+
    \sum_{i\neq j \in \{1,...,m\}} \big(\nu-\delta_{x}\big)^{\otimes (m-2),\downarrow}(\mathrm{d}(z_1,...,z_{i\wedge j-1},z_{i\wedge j+1},...,z_{i\vee j-1},z_{i\vee j+1},...,z_m))\otimes
    \\
    &\hspace{8cm}\otimes\delta_{w}(\mathrm{d}z_i)\otimes\delta_{z}(\mathrm{d}z_j)
    \\
    &\;
    -\sum_{i=1}^{m} \big(\nu-\delta_{x}\big)^{\otimes (m-1),\downarrow}(\mathrm{d}(z_1,...,z_{i-1},z_{i+1},...,z_m))\otimes\delta_{x}(\mathrm{d}z_i)
    \end{aligned}
    \end{equation}
which gives the claim.
\end{proof}

\begin{proof}[Proof of Proposition~\ref{P:generators}] Fix $F=F^{g,m,f}\in\CL$ and $\nu\in{\mathcal N}_f(\mathbb{R}_+)$.
Then there exists $\{M_m;\,m\in\mathbb{N}\}\subseteq\mathbb{N}$ and $L\in\mathbb{N}$ with $F^{g,m,f}\in\CL_{\{M_m;\,m\in\mathbb{N}\},L}$.
We proceed in two steps: \smallskip

\noindent{\em Step~1 (Cell split). } Recall $\Omega^{(r,\theta),\zeta}_{\mbox{\tiny cell split}}$ from (\ref{Eq:PartGener2}), and write
\begin{equation}
\begin{aligned}
    \label{Eq:PartGener2a}
    &\Omega^{(r,\theta),\zeta}_{\mbox{\tiny cell split}} F^{g,m,f}(\nu)
    \\
    &=
    A^{r}_{\mbox{\tiny cell number flow}} F^{g,m,f}(\nu)+A^{(r,\theta),\zeta}_{\mbox{\tiny virus mass split}} F^{g,m,f}(\nu)
    \\
    &=:
    rg\big(\langle 1,\nu\rangle+1\big)\int_{\mathbb{R}^m_+}f\big((z_1,...,z_m)\big)\,\nu^{\otimes m,\downarrow}(\underline{z})
    \\
    &\;
    +r \big(\langle 1,\nu\rangle+1\big)\int_{\mathbb{R}^m_+}f\big((z_1,...,z_m)\big)\,\sum_{k=0}^{\zeta^{-1}z}p^\theta_{k,\zeta^{-1}z}\Big(\big(\nu+\delta_{\zeta k}+\delta_{z_0-\zeta k}-\delta_{z_0}\big)^{\otimes m,\downarrow}-\nu^{\otimes m,\downarrow}\Big)(\mathrm{d}\underline{z})\,\nu(\mathrm{d}z_0).
    \end{aligned}
    \end{equation}\smallskip

    \noindent{\em Step~1.1 (The term $A^r_{\mbox{\tiny cell number flow}} F^{g,m,f}(\nu)$) } This term is constant in $\zeta\in(0,1]$. It therefore trivially follows that for all $F\in{\mathcal L}$,
\begin{equation}
\label{e:0421.1}
   \lim_{\zeta\to 0}\sup_{\nu\in{\mathcal N}_f(\zeta\mathbb{N}_f)}\big|A^r_{\mbox{\tiny cell number flow}}F(\nu)-A^r_{\mbox{\tiny cell number flow}}F(\nu)\big|=0.
\end{equation}\smallskip

\noindent{\em Step~1.2 (The term $A^{(r,\theta),\zeta}_{\mbox{\tiny virus mass split}} F^{g,m,f}(\nu)$) } For all $\nu\in{\mathcal N}_f(\mathbb{R}_+)$,
\begin{equation}
\begin{aligned}
\label{e:043}
   &A^{(r,\theta),\zeta}_{\mbox{\tiny virus mass split}} F^{g,m,f}(\nu)
   \\
   &=
   r\cdot g\big(\langle 1,\nu\rangle+\zeta\big)\int_{\mathbb{R}_+}\int_{\mathbb{R}^m_+}f\big((z_1,...,z_m)\big)\,
   \\
   &\hspace{2cm}
   \sum_{k=0}^{\zeta^{-1}z_0}p^\theta_{k,\zeta^{-1}z_0}
   \Big(\big(\nu-\delta_{z_0}+\delta_{\zeta k}+\delta_{z_0-\zeta k}\big)^{\otimes m,\downarrow}-\nu^{\otimes m,\downarrow}\Big)(\mathrm{d}(z_1,...,z_m))\,\nu(z_0)
\end{aligned}
\end{equation}

Recall from  Lemma~\ref{MR1}(iv) that given $z_0\in\mathrm{supp}(\nu)$,
\begin{equation}
\begin{aligned}
\label{e:048}
   &\sum_{k=0}^{\zeta^{-1}z_0}p^\theta_{k,\zeta^{-1}z_0}
   \Big(\big(\nu-\delta_{z_0}+\delta_{\zeta k}+\delta_{z_0-\zeta k}\big)^{\otimes m,\downarrow}-\nu^{\otimes m,\downarrow}\Big)(\mathrm{d}(z_1,...,z_m))
   \\
   &=
   \sum_{i=1}^m\big(\nu-\delta_{z_0}\big)^{\otimes (m-1),\downarrow}(\mathrm{d}(z_1,...,z_{i-1},z_{i+1},...z_m))
   \otimes\sum_{k=0}^{\zeta^{-1}z_0}p^\theta_{k,\zeta^{-1}z_0}\big(-\delta_{z_0}+\delta_{\zeta k}+\delta_{z_0-\zeta k}\big)(\mathrm{d}z_i)
   \\
   &\;
   +\sum_{i\neq j \in \{1,...,m\}} \big(\nu-\delta_{x}\big)^{\otimes (m-2),\downarrow}(\mathrm{d}(z_1,...,z_{i\wedge j-1},z_{i\wedge j+1},...,z_{i\vee j-1},z_{i\vee j+1},...,z_m))\otimes
    \\
    &\hspace{8cm}\otimes\sum_{k=0}^{\zeta^{-1}z_0}p^\theta_{k,\zeta^{-1}z_0}\delta_{\zeta k}\otimes\delta_{z_0-\zeta k}(\mathrm{d}(z_i,z_j))
\end{aligned}
\end{equation}

By the law of large numbers applied to a binomially distributed random variable, for all Lipschitz functions $f\in{\mathcal C}_b(\mathbb{R}_+)$ with Lipshitz constant $L$ and all $z_0\in\mathbb{R}_+$,
\begin{equation}
\begin{aligned}
\label{e:053}
   \sum_{k=0}^{\zeta^{-1}z_0}p^\theta_{k,\zeta^{-1}z_0}\big|f\big(\zeta k\big)-f\big(\theta z_0\big)\big|
   &\le
   L\sum_{k=0}^{\zeta^{-1}z_0}p^\theta_{k,\zeta^{-1}z_0}\big|\zeta k-\theta z_0\big|
   \\
   &\le
   L\zeta\sqrt{\sum_{k=0}^{\zeta^{-1}z_0}p^\theta_{k,\zeta^{-1}z_0}\big|k-\theta \zeta^{-1}z_0\big|^2}
   \\
   &\le
   L\zeta^{\frac{1}{2}}\sqrt{z_0\theta(1-\theta)}.
\end{aligned}
\end{equation}

Analogously, we can argue that
 \begin{equation}
\begin{aligned}
\label{e:054}
   \sum_{k=0}^{\zeta^{-1}z_0}p^\theta_{k,\zeta^{-1}z_0}\big|f\big(z_0-\zeta k\big)-f\big((1-\theta) z_0\big)\big|
   &\le
   L\zeta^{\frac{1}{2}}\sqrt{z_0\theta(1-\theta)}.
   \end{aligned}
   \end{equation}

   and for all functions $\tilde{f}\in{\mathcal C}_b(\mathbb{R}^2_+)$ of the form $\tilde{f}(x,z_0-x)=f(x)$ for some  Lipshitz function  $f$ with Lipshitz constant $L$ and all $z_0\in\mathbb{R}_+$,
    \begin{equation}
\begin{aligned}
\label{e:055}
   \sum_{k=0}^{\zeta^{-1}z_0}p^\theta_{k,\zeta^{-1}z_0}\big|\tilde{f}\big(\zeta k,z_0-\zeta k\big)-\tilde{f}\big((1-\theta) z_0\big)\big|
   &\le
   L\zeta^{\frac{1}{2}}\sqrt{z_0\theta(1-\theta)}.
   \end{aligned}
   \end{equation}

We therefore find that for all
with $\int z\,\nu(\mathrm{d}z)<\infty$ that
\begin{equation}
\label{e:0421.2}
\begin{aligned}
   &\big|A^{(r,\theta),\zeta}_{\mbox{\tiny virus branching}}F^{g,m,f}(\nu)-A^{(r,\theta)}_{\mbox{\tiny virus branching}}F^{g,m,f}(\nu)\big|
   \\
   &
   \hspace{2cm}\le L r(m^2+1)\zeta^{\frac{1}{2}}\sqrt{\theta(1-\theta)}\big(\sup_{k\in\mathbb{N}}g(k)k^m\big)\int (\sqrt{z_0}\vee 1)\,\nu(\mathrm{d}z_0)
   \\
   &\hspace{2cm}\le L r(m^2+1)\zeta^{\frac{1}{2}}\sqrt{\theta(1-\theta)}\Big(\int z_0\,\nu(\mathrm{d}z_0)\vee M_m\Big).
\end{aligned}
\end{equation}\smallskip

\noindent{\em Step~2 (Virus branching). } Recall $\Omega^{(\sigma,K,\lambda),\zeta}_{\mbox{\tiny virus branching}}$ from (\ref{Eq:PartGener3}). Then
\begin{equation}
\label{e:063}
\begin{aligned}
   \Omega^{(\sigma,K,\lambda),\zeta}_{\mbox{\tiny virus branching}}F^{g,m,f}(\nu)
   =g\big(\langle 1,\nu\rangle\big)\Omega^{(\sigma,K,\lambda),\zeta}_{\mbox{\tiny virus branching}}F^{\1,m,f}(\nu),
\end{aligned}
\end{equation}
while
\begin{equation}
\label{e:056}
\begin{aligned}
   &\Omega^{(\sigma,K,\lambda),\zeta}_{\mbox{\tiny virus branching}}F^{\1,m,f}(\nu)
   \\
   &=
   \big(\zeta^{-1}\sigma+K\big)\zeta^{-1}\int_{\mathbb{R}_+}   z_0\int_{\mathbb{R}^m_+}f\big((z_1,...,z_m)\big)\,\Big(\big(\nu-\delta_{z_0}+ \delta_{z_0+\zeta}\big)^{\otimes m,\downarrow}-\nu^{\otimes m,\downarrow}\nu\Big)(\mathrm{d}(z_1,...,z_m))\, \nu(\mathrm{d}z_0)
    \\
    &\,+ \zeta^{-1} \int_{\mathbb{R}_+} \big(\zeta^{-1}\sigma +\lambda (z_0-\zeta)\big)z_0 \int_{\mathbb{R}^m_+}f\big((z_1,...,z_m)\big)\,\Big(\big(\nu-\delta_{z_0}+ \delta_{z_0-\zeta}\big)^{\otimes m,\downarrow}-\nu^{\otimes m,\downarrow}\Big)(\mathrm{d}(z_1,...,z_m))\, \nu(\mathrm{d}z_0)
    \\
    &=
    \zeta^{-2}\sigma \int_{\mathbb{R}_+}   z_0\int_{\mathbb{R}^m_+}f\big((z_1,...,z_m)\big)\,\Big(\big(\nu-\delta_{z_0}+ \delta_{z_0+\zeta}\big)^{\otimes m,\downarrow}+\big(\nu-\delta_{z_0}+ \delta_{z_0-\zeta}\big)^{\otimes m,\downarrow}-2\nu^{\otimes m,\downarrow}\Big)(\mathrm{d}(z_1,...,z_m))\, \nu(\mathrm{d}z_0)
    \\
    &\;
    +\zeta^{-1}K\int_{\mathbb{R}_+}   z_0\int_{\mathbb{R}^m_+}f\big((z_1,...,z_m)\big)\,\Big(\big(\nu-\delta_{z_0}+ \delta_{z_0+\zeta}\big)^{\otimes m,\downarrow}-\nu^{\otimes m,\downarrow}\Big)(\mathrm{d}(z_1,...,z_m))\, \nu(\mathrm{d}z_0)
    \\
    &\;
    +\zeta^{-1}\lambda \int_{\mathbb{R}_+}  (z_0-\zeta) z_0 \int_{\mathbb{R}^m_+}f\big((z_1,...,z_m)\big)\,\Big(\big(\nu-\delta_{z_0}+ \delta_{z_0-\zeta}\big)^{\otimes m,\downarrow}-\nu^{\otimes m,\downarrow}\Big)(\mathrm{d}(z_1,...,z_m))\, \nu(\mathrm{d}z_0)
    \\
    &=:
    \Omega^{\sigma,\zeta}_{\mbox{\tiny natural branching}}F^{\1,m,f}(\nu)+\Omega^{K,\zeta}_{\mbox{\tiny extra birth}}F^{\1,m,f}(\nu)+\Omega^{\lambda,\zeta}_{\mbox{\tiny competition}}F^{\1,m,f}(\nu).
\end{aligned}
\end{equation}

We will once more treat the three terms separately. \smallskip

\noindent{Step~2.1 (Natural branching)} Use that by Lemma~\ref{MR1}(i), for all $\nu\in{\mathcal N}_f(\mathbb{R}_+)$ and $z_0\in\supp(\nu)$,
\begin{equation}
\label{e:057}
\begin{aligned}
   &\Big(\big(\nu-\delta_{z_0}+\delta_{z_0\pm\zeta}\big)^{\otimes m,\downarrow}-\big(\nu-\delta_{z_0}\big)^{\otimes m,\downarrow}\Big)(\mathrm{d}(z_1,...,z_m))
   \\
   &=
   \sum_{i=1}^m\big(\nu-\delta_{z_0}\big)^{\otimes (m-1),\downarrow}(\mathrm{d}(z_1,...,z_{i-1},z_{i+1},...,z_m))\otimes\delta_{z_0\pm\zeta}(\mathrm{d}z_i),
\end{aligned}
\end{equation}
together with
\begin{equation}
\label{e:058}
\begin{aligned}
   &\Big(\big(\nu-\delta_{z_0}\big)^{\otimes m,\downarrow}-\nu^{\otimes m,\downarrow}\Big)(\mathrm{d}(z_1,...,z_m))
   \\
   &=
   -\sum_{i=1}^m\big(\nu-\delta_{z_0}\big)^{\otimes (m-1),\downarrow}(\mathrm{d}(z_1,...,z_{i-1},z_{i+1},...,z_m))\otimes\delta_{z_0}(\mathrm{d}z_i),
\end{aligned}
\end{equation}
to conclude that
\begin{equation}
\label{e:059}
\begin{aligned}
  &\Big(\big(\nu-\delta_{z_0}+\delta_{z_0+\zeta}\big)^{\otimes m,\downarrow}+\big(\nu-\delta_{z_0}+\delta_{z_0-\zeta}\big)^{\otimes m,\downarrow}-2\nu^{\otimes m,\downarrow}\Big)(\mathrm{d}(z_1,...,z_m))
  \\
  &=
  \sum_{i=1}^m\big(\nu-\delta_{z_0}\big)^{\otimes (m-1),\downarrow}(\mathrm{d}(z_1,...,z_{i-1},z_{i+1},...,z_m))\otimes\Big(\delta_{z_0+\zeta}+\delta_{z_0-\zeta}-2\delta_{z_0}\Big)(\mathrm{d}z_i)
\end{aligned}
\end{equation}
and
\begin{equation}
\label{e:051}
\begin{aligned}
&\int_{\mathbb{R}_+}z_0\int_{\mathbb{R}^m_+}f\big((z_1,...,z_m)\big)\Big(\big(\nu-\delta_{z_0}+\delta_{z_0+\zeta}\big)^{\otimes m,\downarrow}+\big(\nu-\delta_{z_0}+\delta_{z_0-\zeta}\big)^{\otimes m,\downarrow}-2\nu^{\otimes m,\downarrow}\Big)(\mathrm{d}(z_1,...,z_m))\nu(dz_0)
\\
&=
\int_{\mathbb{R}_+}z_0\int_{\mathbb{R}^m_+}f\big((z_1,...,z_m)\sum_{i=1}^m\big(\nu-\delta_{z_0}\big)^{\otimes (m-1),\downarrow}(\mathrm{d}(z_1,...,z_{i-1},z_{i+1},...,z_m))\otimes
\\
&\hspace{9cm}
\otimes\Big(\delta_{z_0+\zeta}+\delta_{z_0-\zeta}-2\delta_{z_0}\Big)(\mathrm{d}z_i)
\nu(dz_0)
\\
&=
\int_{\mathbb{R}^m_+}\sum_{i=1}^m\Big(\big(z_i-\zeta\big)f\big((z_1,...,z_{i-1},z_i+\zeta,z_{i+1},...,z_m)
+\big(z_i+\zeta\big)f\big((z_1,...,z_{i-1},z_i-\zeta,z_{i+1},...,z_m)\big)
\\
&\;
\hspace{3cm}-2z_if\big((z_1,...,z_m)\big)\Big)\,\nu^{\otimes m,\downarrow}(\mathrm{d}(z_1,...,z_m)).
\end{aligned}
\end{equation}

Thus for all $\nu\in{\mathcal N}_f(\zeta\mathbb{N})$, by Taylor expansion around $z_i$,
\begin{equation}
\label{e:061}
\begin{aligned}
&\Omega^{\sigma,\zeta}_{\mbox{\tiny natural branching}}F^{\1,m,f}(\nu)
\\
&=
\zeta^{-2}\sigma \int_{\mathbb{R}^m_+}\sum_{i=1}^m\Big((z_i+\zeta)f\big((z_1,...,z_{i-1},z_i+\zeta,z_{i+1},...,z_m)\big)
+(z_i-\zeta)f\big((z_1,...,z_{i-1},z_i-\zeta,z_{i+1},...,z_m)\big)-
\\
&\hspace{6cm}
-2z_i f\big((z_1,...,z_m)\big)\Big)\,\nu^{\otimes m,\downarrow}(\mathrm{d}(z_1,...,z_m))
\\
&\tzetao
\sigma\int_{\mathbb{R}^m_+}\sum_{i=1}^m z_i\frac{\partial^2}{\partial (z_i)^2}f\big((z_1,...,z_m)\big)\,\nu^{\otimes m,\downarrow}(\mathrm{d}(z_1,...,z_m))
\\
&=:
\Omega^{\sigma}_{\mbox{\tiny natural branching}}F^{\1,m,f}(\nu).
\end{aligned}
\end{equation}

Moreover, by the Taylor theorem,
\begin{equation}
\label{e:060}
\begin{aligned}
&\big|\Omega^{\sigma,\zeta}_{\mbox{\tiny natural branching}}F^{\1,m,f}(\nu)-\Omega^{\sigma}_{\mbox{\tiny natural branching}}F^{\1,m,f}(\nu)\big|
\\
&\le
\zeta\sigma\langle 1,\nu\rangle^m Lm,
\end{aligned}
\end{equation}
and therefore
\begin{equation}
\label{e:064}
\begin{aligned}
\big|\Omega^{\sigma,\zeta}_{\mbox{\tiny natural branching}}F^{g,m,f}(\nu)-\Omega^{\sigma}_{\mbox{\tiny natural branching}}F^{g,m,f}(\nu)\big|
&\le
4\zeta\cdot \sigma M_m L m.
\end{aligned}
\end{equation}
\smallskip

\noindent{Step~2.2 (Extra birth)} Use that by (\ref{e:057}) together with (\ref{e:058}), for all $\nu\in{\mathcal N}_f(\mathbb{R}_+)$ 
\begin{equation}
\label{e:050}
\begin{aligned}
&\Omega^{K,\zeta}_{\mbox{\tiny extra birth}}F^{g,m,f}(\nu)
\\
&=
g\big(\langle 1,\nu\rangle\big)\zeta^{-1}K\int_{\mathbb{R}_+}   z_0\int_{\mathbb{R}^m_+}f\big((z_1,...,z_m)\big)\,\Big(\big(\nu-\delta_{z_0}+ \delta_{z_0+\zeta}\big)^{\otimes m,\downarrow}-\nu^{\otimes m,\downarrow}\Big)(\mathrm{d}(z_1,...,z_m))\, \nu(\mathrm{d}z_0)
\\
&=
g\big(\langle 1,\nu\rangle\big)\zeta^{-1}K\int_{\mathbb{R}_+}   z_0
\int_{\mathbb{R}^m_+}\sum_{i=1}^m f\big((z_1,...,z_m)\big)\,\nu^{\otimes (m-1),\downarrow}(\mathrm{d}(z_1,...,z_{i-1},z_{i+1},...,z_m))\otimes\big(\delta_{z_0+\zeta}-\delta_{z_0}\big)(\mathrm{d}z_i)\,
\nu(\mathrm{d}z_0)
\\
&=
g\big(\langle 1,\nu\rangle\big)\zeta^{-1}K
\int_{\mathbb{R}^m_+}\sum_{i=1}^m \Big(\big(\big(z_i-\zeta\big)f\big((z_1,...,z_{i-1},z_i+\zeta,z_{i+1},...,z_m)\big)-z_i f\big((z_1,...,z_m)\big)\Big)\,\nu^{\otimes m,\downarrow}(\mathrm{d}\underline{z})
\\
&\tzetao
g\big(\langle 1,\nu\rangle\big)K\int_{\mathbb{R}^m_+}\sum_{i=1}^m z_i\frac{\partial}{\partial z_i}f\big((z_1,...,z_m)\big)\,\nu^{\otimes m,\downarrow}(\mathrm{d}\underline{z})
\\
&=:
\Omega^{K}_{\mbox{\tiny extra birth}}F^{g,m,f}(\nu).
\end{aligned}
\end{equation}

Following the same line of argument which lead to (\ref{e:064}) yields for all $\nu\in{\mathcal N}_f(\zeta\mathbb{N})$ that
\begin{equation}
\label{e:052}
\begin{aligned}
\big|\Omega^{K,\zeta}_{\mbox{\tiny extra birth}}F^{g,m,f}(\nu)-\Omega^{K}_{\mbox{\tiny extra birth}}F^{g,m,f}(\nu)\big|
&\le
2\zeta\cdot K M_m L m.
\end{aligned}
\end{equation}
\smallskip

\noindent{Step~2.3 (Competition) } By (\ref{e:057}) together with (\ref{e:058}), for all $\nu\in{\mathcal N}_f(\mathbb{R}_+)$ 
\begin{equation}
\label{e:065}
\begin{aligned}
&\Omega^{\lambda,\zeta}_{\mbox{\tiny competition} }F^{g,m,f}(\nu)
\\
&=
g\big(\langle 1,\luis{\nu}\rangle\big)\zeta^{-1}\lambda\int_{\mathbb{R}_+}z_0(z_0-\zeta)\int_{\mathbb{R}^m_+}f\big((z_1,...,z_m)\big)\,\Big(\big(\nu-\delta_{z_0}+ \delta_{z_0-\zeta}\big)^{\otimes m,\downarrow}-\nu^{\otimes m,\downarrow}\Big)(\mathrm{d}(z_1,...,z_m))\, \nu(\mathrm{d}z_0)
\\
&=
g\big(\langle 1,\nu\rangle\big)\zeta^{-1}\lambda\int_{\mathbb{R}_+}   z_0(z_0-\zeta)
\int_{\mathbb{R}^m_+}\sum_{i=1}^m f\big((z_1,...,z_m)\big)\,\nu^{\otimes (m-1),\downarrow}(\mathrm{d}(z_1,...,z_{i-1},z_{i+1},...,z_m))\otimes
\\
&\hspace{10cm}
\otimes\big(\delta_{z_0-\zeta}-\delta_{z_0}\big)(\mathrm{d}z_i)\,
\nu(\mathrm{d}z_0)
\\
&=
g\big(\langle 1,\nu\rangle\big)\zeta^{-1}\lambda
\int_{\mathbb{R}^m_+}\sum_{i=1}^m \Big(\big(\big(z_i+\zeta\big)z_i f\big((z_1,...,z_{i-1},z_i-\zeta,z_{i+1},...,z_m)\big)-z_i f\big((z_1,...,z_m)\big)\Big)\,\nu^{\otimes m,\downarrow}(\mathrm{d}\underline{z})
\\
&\tzetao
-g\big(\langle 1,\nu\rangle\big)\lambda\int_{\mathbb{R}^m_+}\sum_{i=1}^m z_i^2\frac{\partial}{\partial z_i}f\big((z_1,...,z_m)\big)\,\nu^{\otimes m,\downarrow}(\mathrm{d}\underline{z})
\\
&=:
\Omega^{\lambda}_{\mbox{\tiny competition}}F^{g,m,f}(\nu).
\end{aligned}
\end{equation}

Following the same line of argument which lead to (\ref{e:064}) yields for all $\nu\in{\mathcal N}_f(\zeta\mathbb{N})$ that
\begin{equation}
\label{e:066}
\begin{aligned}
\big|\Omega^{\lambda,\zeta}_{\mbox{\tiny competition}}F^{g,m,f}(\nu)-\Omega^{\lambda}_{\mbox{\tiny competition}}F^{g,m,f}(\nu)\big|
&\le
2\zeta\cdot K M_m L m.
\end{aligned}
\end{equation}
This finishes the proof of Proposition~\ref{P:generators}.
\end{proof}

We collect here what we have proven in the direction of establishing the well-posed martingale problem stated in Theorem~\ref{T:001}.
\begin{proposition}[existence of a solution]
Let the family $\{Z^\zeta;\,\zeta\in(0,1]\}$ be as define as in Proposition \ref{Prop:tightness}. Assume additionally that $Z$ is a limit point
of $\{Z^\zeta;\,\zeta\in(0,1]\}$. Then $Z$ is a solution of the $(\Omega_{\mbox{\tiny $2$-level}},{\mathcal D}(\Omega_{\mbox{\tiny $2$-level}}),P)$-martingale problem, where $P$ is the law of $Z_0$.
\label{Prop:Conv}
\end{proposition}

\begin{proof}
The fact that $Z$ solves the $(\Omega_{\mbox{\tiny $2$-level}},{\mathcal L},P)$-martingale problem follows from \cite[Theorem~3.6.3]{EthierKurtz1986}. The stronger claim that $Z$ solves the $(\Omega_{\mbox{\tiny $2$-level}},{\mathcal D}(\Omega_{\mbox{\tiny $2$-level}}),P)$-martingale problem follows from Proposition~\ref{P:003}.
\end{proof}

\section{Uniqueness of the martingale problem}
\label{S:uniqueness}
In this section we establish a dual relation for our 2-level branching model. Such a duality relation will then imply the uniqueness of the martingale problem and thus allow to finish the proof of Theorem~\ref{T:001}.

\subsection{Duality relations for the Yule process}
\label{SuB:Yuledual}
Recall our $2$-level branching process with cell division, $Z=(Z_t)_{t\ge 0}$, from Definition~\ref{Def:006}. Put $W_t:=\langle 1,Z_t\rangle$ for $t\ge 0$. Then the process
$W:=(W_t)_{t\ge 0}$ is a pure Markovian jump process on $\mathbb{N}$ which jumps from $k$ to $k+1$ at the cell splitting rate $r$.
In the literature $W$ is often referred to as the {\em Yule process}.

As a preparation for our duality relation, we present some identities for $W$ in this subsection.

\begin{lemma}[distribution of the Yule process] For all $t\ge 0$ and $z\in[0,1]$,
\begin{equation}
\label{e:030}
\mathbb{E}_w\big[z^{W_t}\big]=z_t^{w},
\end{equation}
where $(z_t)_{t\ge 0}$ is the unique solution of the initial value problem $\frac{\mathrm{d}}{\mathrm{d}t}z_t=-rz_t(1-z_t)$ and $z_0=z$.

In particular, if $X(\mu)$ denotes an exponential random variable with mean $\mu^{-1}$, then the following identity holds:
\begin{equation}
\label{e:036}
\mathbb{E}_w\big[z^{W_t}\big]=\mathbb{E}\big[e^{-\frac{1-z}{z}w X(e^{-rt})}\big].
\end{equation}
\label{L:001}
\end{lemma}

\begin{proof} $W$ is a Markov process whose generator $\Omega_{\mbox{\tiny Yule}}$ acts on bounded functions $f:\mathbb{N}\to\mathbb{R}$ as
follows:
\begin{equation}
\label{e:031}
\Omega^r_{\mbox{\tiny Yule}}f(k)=rk\big(f(k+1)-f(k)\big).
\end{equation}

Fix $z\in(0,1]$ and apply (\ref{e:031}) to $f_z(k):=z^k$. Then
\begin{equation}
\label{e:032}
\Omega^r_{\mbox{\tiny Yule}}f_z(k)=rk\big(f_z(k+1)-f_z(k)\big)=-rz(1-z)kz^{k-1}=-rz(1-z)\frac{\partial}{\partial z}f_z(k).
\end{equation}

This yields the duality relation (\ref{e:030}). Notice that the initial value problem  is solved by
\begin{equation}
\label{e:033}
z_t=\frac{1}{1+(z^{-1}-1)e^{rt}}=\frac{e^{-rt}}{(z^{-1}-1)+e^{-rt}}=\frac{e^{-rt}}{\frac{1-z}{z}+e^{-rt}}=\mathbb{E}\big[e^{-\frac{1-z}{z}X(e^{-rt})}
\big],
\end{equation}
which gives the claim.
\end{proof}

From here we derive immediately the well-known scaling result:
\begin{corollary}[scaling limit of the  Yule process] Let $W=(W_t)_{t\ge 0}$ be the rate $r$-Yule process starting in $w\in\mathbb{N}$. Then
\begin{equation}
\label{e:046}
e^{-rt}W_t\Tto\Gamma(w,1),
\end{equation}
where $\Gamma(w,1)$ is a Gamma-distributed random variable with parameters $w$ and $1$.
\label{Cor:002}
\end{corollary}

\begin{proof} By the branching property we may assume w.l.o.g.\ that $w=1$. Put $\hat{W}_t:=e^{-rt}W_t$ for $t\ge 0$. Then by (\ref{e:036}) (applied with $z=e^{-\lambda}$)
for all $\lambda>0$,
\begin{equation}
\label{e:067}
\begin{aligned}
\mathbb{E}\big[e^{-\lambda \hat{W}_t}\big]
&=\mathbb{E}\big[\exp\big(-\frac{1-e^{-\lambda e^{-rt}}}{e^{-\lambda e^{-rt}}}X(e^{-rt})\big)\big]
\\
&\tro
\mathbb{E}\big[\exp\big(-\lambda X(1)\big)\big].
\end{aligned}
\end{equation}
As Laplace transforms are convergence determining, the claim follows.
\end{proof}

\begin{corollary}[a moment dual]
Let $W$ be the Yule process with parameter $r>0$ started in  $w\in \N$. Then the next
moment duality holds for all $m\in\mathbb{N}$,
\begin{equation}
\label{e:068}
\E_w[W_t(W_t+1)\cdot ...\cdot (W_t+m-1)]=w(w+1)\cdot ...\cdot(w+m-1) e^{rmt}.
\end{equation}
\end{corollary}

\begin{proof} Consider the dual function
\begin{equation}
\label{e:069}
H_{m,s}\big(w,t\big)=w(w+1)\cdot ...\cdot (w+m-1)e^{mr(s+t)}
\end{equation}

Then
\begin{equation}
\label{e:070}
\begin{aligned}
\Omega^r_{\mbox{\tiny Yule}}H_{m,s}\big(w,t\big)
&=
rw\Big(H_{m,s}\big(w+1,t\big)-H_{m,s}\big(w,t\big)\Big)
\\
&=
rw(w+1)\cdot ...\cdot (w+m-1)\ \big((w+m)-w\big)e^{mr(s+t)}
\\
&=
rm H_{m,s}(w,t),
\end{aligned}
\end{equation}
which implies the Feynman-Kac duality relation:
\begin{equation}
\label{e:071}
\mathbb{E}_w\big[W_t(W_t+1)\cdot ...\cdot (W_t+m-1)e^{rmt}\big]=\mathbb{E}_t\big[w(w+1)\cdot ...\cdot (w+m-1)e^{rm(t+s)}\big],
\end{equation}
or equivalently, the claim in (\ref{e:069}).
\end{proof}

\subsection{The duality relation}
\label{Sub:duality}
In this subsection we state
a novel Feynman-Kac duality. It combines ideas from (\cite{HutzenthalerWakolbinger2007,GrevenSturmWinterZaehle,HermannPfaffelhuber2020}) obtained for samples of logistic branching with and without disasters.

The dual process is a Markov process $K:=(q_t,M_t,X_t)_{t\ge 0}$ with state space $(0,1]\times\mathbb{K}$ where
\begin{equation}
\label{ASAW-b:111}
\mathbb{K}:=\bigcup_{m\in\mathbb{N}}\{m\}\times\mathbb{R}_+^{m}
\end{equation}
and with the following dynamics: given the current state $(q,m,(x_1,...,x_m))$
\begin{itemize}
\item $q$ follows the ordinary differential equation
\begin{equation}
\label{e:035}
\frac{\mathrm{d}}{\mathrm{d}t}q_t=-r q_t(1-q_t).
\end{equation}
\item For all $k=1,...,m$ the following jump happens at rate $qr$:
\begin{equation}
\label{ASAW-b:025}
\big(q,m,(x_1,...,x_m)\big)\mapsto\big(q,m,(x_1,...,x_{k-1},\theta x_k,x_{k+1},...,x_m)\big).
\end{equation}
\item For all $k=1,...,m$ the following jump happens at rate $qr$:
\begin{equation}
\label{ASAW-b:026}
\big(q,m,(x_1,...,x_m)\big)\mapsto\big(q,m,(x_1,...,x_{k-1},(1-\theta)x_k,x_{k+1},...,x_m)\big).
\end{equation}
\item For all $1\le k_1\not =k_2\le m$ the following jump happens at rate $qr$:
\begin{equation}
\label{ASAW-b:027}
\begin{aligned}
&\big(q,m,(x_1,...,x_m)\big)
\\
&\mapsto\big(q,m-1,(x_1,...,x_{k_1-1},\theta x_{k_1}+(1-\theta)x_{k_2},x_{k_1+1},...,x_{k_2-1},x_{k_2+1},...,x_m)\big).
\end{aligned}
\end{equation}
\item For all $1\le k_1\not =k_2\le m$ the following jump happens at rate $qr$:
\begin{equation}
\label{ASAW-b:027b}
\begin{aligned}
&\big(q,m,(x_1,...,x_m)\big)
\\
&\mapsto
\big(q,m-1,(x_1,...,x_{k_1-1},(1-\theta) x_{k_1}+\theta x_{k_2},x_{k_1+1},...,x_{k_2-1},x_{k_2+1},...,x_m)\big).
\end{aligned}
\end{equation}
\item In between two jumps of $M$, the coordinate processes $X_k$, $k=1,...,m$, perform independent logistic Feller branching diffusion
\begin{equation}
\label{ASAW-B:028}
\mathrm{d}X_k(t)=X_k(t)(K-\sigma X_k(t))\mathrm{d}t+\sqrt{2\lambda X_k(t)\mathrm{d}}B_k(t),\hspace{.2cm}t\ge 0,
\end{equation}
where the $B_k$ for $k=1,...,m$ denote independent Brownian motions, i.e., the strong Markov process which generator $(\Omega_{\mbox{\tiny log-Feller}},{\mathcal D}(\Omega_{\mbox{\tiny log-Feller}}))$ acts on $f\in{\mathcal C}^2(\mathrm{R}_+)\subseteq{\mathcal D}(\Omega_{\mbox{\tiny log-Feller}}))$ as
\begin{equation}
\label{ASAW-B:028b}
\Omega_{\mbox{\tiny log-Feller}}f(x)=-\Psi(x)f'(x)+\lambda xf''(x),\hspace{.2cm}x\in\mathbb{R}_+
\end{equation}
with the branching mechanism $\Psi(x):=-K x+\sigma x^2$.
\end{itemize}

The dynamics that we have just defined can be described by the generator acting on functions in ${\mathcal B}\big({\mathcal N}_f(\mathbb{R}_+)\times(0,1]\times\mathbb{K}\big)$ that are twice continuously differentiable in the $\underbar x$-variables and once continuously differentiable in $q$ variable.  as follows:
\begin{equation}
\begin{aligned}
\label{Eq:GeneratorX_t}
&\Omega^\ast_{\mbox{\tiny dual}}
= \Omega_{\mbox{\tiny cell number flow}}^{r,\ast}+\Omega_{\mbox{\tiny disaster}}^{(r,\theta),\ast}+ \Omega_{\mbox{\tiny virus branching}}^{(\lambda,K,\sigma),\ast}
\end{aligned}
\end{equation}
with
\begin{equation}
\label{e:072}
\Omega_{\mbox{\tiny cell number flow}}^{r,\ast} F(q,m,\underbar x)
=:-r(1-q)q\frac{\partial}{\partial q} F(q,m,\underbar x)
\end{equation}
and
\begin{equation}
\label{e:073}
\begin{aligned}
&\Omega_{\mbox{\tiny disaster}}^{(r,\theta),\ast} F(q,m,\underbar x)
\\
&=
+qr\sum_{i=1}^{m} \big(F(q,m,(x_1,...,\theta x_i,..., x_m))+ F(q,m,(x_1,...,(1-\theta) x_i,..., x_m))-2F(q,m,\underbar{x})\big)
\\
&\;
+qr\sum_{1\leq i<j\leq m}\big(F(q,m-1,(x_1,...,x_{j-1}, \theta x_{j}+(1-\theta)x_i,x_{j+1},..., x_{i-1}, x_{i+1},..., x_m))-F(q,m,\underbar x)\big)
\\
&\;
+qr\sum_{1\leq i<j\leq m}\big(F(q,m-1,(x_1,...,x_{j-1}, x_{j+1},...,x_{i-1},\theta x_{j}+(1-\theta)x_i, x_{i+1},..., x_m))-F(q,m,\underbar x)\big),
\end{aligned}
\end{equation}
and
\begin{equation}
\label{e:074}
\begin{aligned}
\Omega_{\mbox{\tiny virus branching}}^{(\lambda,K,\sigma),\ast} F(q,m,\underbar x)
&=
\sum_{i=1}^{m} (K x_i-\sigma x_i^2)\frac{\partial }{\partial x_j} F(q,m,\underbar x)
+\sum_{i=1}^{m} \lambda x_i\frac{\partial^2 }{\partial x_i^2}F(q,m,\underbar x).
\end{aligned}
\end{equation}

Consider the {\em dual function}  $F:{\mathcal N}_f(\mathbb{R}_+)\times (0,1]\times\mathbb{K}\to \mathbb{R}$ defined by:
\begin{equation}
\label{ASAW-b:024}
F\big(\nu,(q,m,(x_1,...,x_m))\big)=q^{\langle 1,\nu\rangle}\cdot H\big(\nu,(m,(x_1,...,x_m))\big)
\end{equation}
with
\begin{equation}
\label{ASAW-b2:024}
H\big(\nu,(m,(x_1,...,x_m))\big)=\int e^{-\sum_{k=1}^m x_k z_k}\,\nu^{\otimes m,\downarrow}(\mathrm{d}\underline{z}).
\end{equation}

We have the following duality relation:
\begin{theorem}[Feynman-Kac duality relation]
Let $Z$ be $2$-level branching model with cell-division and logistic virus branching diffusion, and $K=(q,M,(X_1,...,X_M))$ the dual process defined above.
Then for all $\nu\in{\mathcal N}_f(\mathbb{R}_+)$, $m\in\mathbb{N}$, $q\in(0,1]$ and $(x_1,...x_m)\in\mathbb{R}_+$ 
\begin{equation}
\begin{aligned}
\label{ASAW-duality}
&\mathbb{E}_{\nu}\big[q^{\langle 1,\nu_t\rangle}\int e^{-\sum_{k=1}^{m} z_k x_k}\,Z_t^{\otimes m,\downarrow}(\mathrm{d}\underline{z})\big]
\\
&=
\mathbb{E}_{(q,m,(x_1,x_2,...,x_m))}\Big[e^{r\int_{0}^{t} q_sM_s^2\,\mathrm{d}s}\,q_t^{\langle 1,\nu\rangle}
\int e^{-\sum_{k=1}^{M_t} z_k X_{k}(t)}\,\nu^{\otimes M_t,\downarrow}(\mathrm{d}\underline{z})\Big].
\end{aligned}
\end{equation}
\label{T:003}
\end{theorem}

\begin{proof} We divide the proof in two steps. Recall $\Omega^{(r,\theta)}_{\mbox{\tiny cell split}}$ from (\ref{y:005}).
In the first step we show that this generator is dual to
$\Omega_{\mbox{\tiny cell number flow}}^{r,\ast}+\Omega_{\mbox{\tiny disaster}}^{(r,\theta),\ast}$ plus a potential giving rise to the Feynman-Kac term.
In the second step we show that the generator
$\Omega^{(\sigma,K,\lambda),\ast}_{\mbox{\tiny virus branching}}$ is dual to $\Omega^{(\lambda,K,\sigma),\ast}_{\mbox{\tiny virus branching}}$.

Recall the dual function $F:{\mathcal N}_f(\mathbb{R}_+)\times(0,1]\times\mathbb{K}\to\mathbb{R}$ from (\ref{ASAW-b:024}).
\smallskip

\noindent{\em Step~1 } We start with the generators which describe the division cell in the forward and the merging  of cells in the backward picture. We here show that
\begin{equation}
\label{e:077}
\begin{aligned}
&\Omega^{(r,\theta)}_{\mbox{\tiny cell split}}F\big(\nu,(q,m,\underbar x)\big)
\\
&=
\Omega_{\mbox{\tiny cell number flow}}^{r,\ast}F\big(\nu,(q,m,\underbar x)\big)+\Omega_{\mbox{\tiny disaster}}^{(r,\theta),\ast}F\big(\nu,(q,m,\underbar x)\big)
+m^2qrF\big(\nu,(q,m,\underbar x)\big).
\end{aligned}
\end{equation}

Applying Lemma~\ref{MR1}(iv)  it is clear that
\begin{equation}
\begin{aligned}
\label{e:015}
&\big(\nu-\delta_{z_0}+\delta_{z_0\theta}+\delta_{z_0(1-\theta)}\big)^{\otimes m,\downarrow}(\mathrm{d}(z_1,...,z_m))
-\nu^{\otimes m,\downarrow}(\mathrm{d}(z_1,...,z_m))
\\
&= \sum_{i=1}^{m}\big(\nu-\delta_{z_0}\big)^{\otimes (m-1),\downarrow}(\mathrm{d}(z_1,...,z_{i-1},z_{i+1},...z_m))\otimes\big(\delta_{z_0\theta}+\delta_{z_0(1-\theta)}-\delta_{z_0}\big)(\mathrm{d}z_i)
\\
&\;
+\sum_{1\le i\not =j\le m}\big(\nu-\delta_{z_0}\big)^{\otimes (m-2),\downarrow}(\mathrm{d}(z_1,...,z_{i\wedge j-1},z_{i\wedge j+1},...,z_{i\vee j-1},z_{i\vee j+1},...,z_m)\otimes\delta_{z_0(1-\theta)}(\mathrm{d}z_i)
\otimes\delta_{z_0\theta}(\mathrm{d}z_j).
\end{aligned}
\end{equation}

In what follows we write $f_{\underbar x}(z_1,...,z_{m}):=e^{-\sum_{i=1}^mx_iz_i}$, and use  for $i,j\in\{1,...,m\}$ with $i\not =j$ and $\alpha\in(0,1]$ the
abbreviations:
\begin{equation}
\label{e:075}
\underbar x_\alpha^{(i)}=(x_1,...,x_i\alpha, ...,x_m)
\end{equation}
and
\begin{equation}
\label{e:076}
\underbar x_{\alpha,\beta}^{(i,j)}=(x_1,...,x_{i\wedge j-1},x_i\alpha+x_j\beta,x_{i\wedge j+1}...,x_{i\vee j-1},x_{i\vee j+1},..., x_m).
\end{equation}

Then
\begin{equation}
\begin{aligned}
\label{e:016}
&H\big(\nu-\delta_{z_0}+\delta_{z_0\theta}+\delta_{(1-\theta)z_0},(\underbar x,m)\big)-H\big(\nu,(m,\underbar x)\big)
\\
&=\int_{\mathbb{R}_+^m} f_{\underbar x}(z_1,...,z_{m})\,\Big(\big(\nu-\delta_{z_0}+\delta_{z_0\theta}+\delta_{z_0(1-\theta)}\big)^{\otimes m,\downarrow}-\nu^{\otimes m,\downarrow}\Big)(\mathrm{d}(z_1,...,z_{m}))
\\
&=
\int_{\mathbb{R}_+^{(m-1)}}\sum_{i=1}^{m} \big(f_{\underbar x_{\theta}^{(i)}}+f_{\underbar x_{(1-\theta)}^{(i)}}-f_{\underbar x}\big)(z_1,...,z_{i-1},z_0,z_{i+1},...,z_m)\,\big(\nu-\delta_{z_0}\big)^{\otimes (m-1),\downarrow}(\mathrm{d}(z_1,...,z_{i-1},z_{i+1},...z_{m}))
\\
&\,
+\int_{\mathbb{R}_+^{(m-2)}}\sum_{1\le i\not =j\le m}  f_{\underbar x_{1-\theta,\theta}^{(i,j)}}(z_1,...,z_{i\wedge j-1},z_0,z_{i\wedge j+1},...,z_{i\vee j-1},z_{i\vee j+1},..., z_m)
\\
&\hspace{3cm}
\big(\nu-\delta_{z_0}\big)^{\otimes (m-2),\downarrow} (\mathrm{d}(z_1,...,z_{i\wedge j-1},z_{i\wedge j+1},...,z_{i\vee j-1},z_{i\vee j+1},...,z_{m})).
\end{aligned}
\end{equation}

Thus
\begin{equation}
\begin{aligned}
\label{e:014}
&\Omega^{(r,\theta)}_{\mbox{\tiny cell split}}  q^{\langle \1,\nu\rangle} H(\nu,(m,\underbar x))
\\
&=
r\int_{\mathbb{R}_+} \Big(q^{\langle 1,\nu\rangle+1} H(\nu-\delta_{z_0}+\delta_{z_0\theta}+\delta_{(1-\theta)z_0},(m,\underbar x))-q^{\langle 1,\nu\rangle} H\big(\nu,(m,\underbar x)\big)\Big)\,\nu(\mathrm{d}z_0)
\\
&=
rq^{\langle 1,\nu\rangle+1}\int\Big(H(\nu-\delta_{z_0}+\delta_{z_0\theta}+\delta_{(1-\theta)z_0},(\underbar x,m))-H\big(\nu,(m,\underbar x)\big)\Big)
\,\nu(\mathrm{d}z_0)
\\
&\;+r\langle 1,\nu\rangle\Big(q^{\langle 1,\nu\rangle+1}-q^{\langle 1,\nu\rangle}\Big) H(\nu,(m,\underbar x))
\\
&=
q^{\langle 1,\nu\rangle+1}\Omega^{(r,\theta)}_{\mbox{\tiny cell split}}  H(\nu,(m,\underbar x))+\Omega^{r,\ast}_{\mbox{\tiny cell number flow}}F\big(\nu,(q,m,\underbar x)\big)
\end{aligned}
\end{equation}

Applying (\ref{e:016}) implies that
\begin{equation}
\begin{aligned}
\label{e:014b}
&q^{\langle 1,\nu\rangle+1}\Omega^{(r,\theta)}_{\mbox{\tiny cell split}}  H(\nu,(m,\underbar x))
\\
&=
qr\cdot q^{\langle 1,\nu\rangle}\sum_{i=1}^{m} H\big(\nu,(m,\underbar x_\theta^{(i)})\big)
+qr\cdot q^{\langle 1,\nu\rangle}\sum_{i=1}^{m} H\big(\nu,(m,\underbar x_{(1-\theta)}^{(i)})\big)
\\
&\;
+qr\cdot q^{\langle 1,\nu\rangle}\sum_{i\neq j \in \{1,...,m\}} H\big(\nu,(m-1,\underbar x_{\theta,1-\theta}^{(i,j)})\big)
-qm r\cdot q^{\langle 1,\nu\rangle} H\big(\nu,(m,\underbar x)\big)
\\
&=
qr\cdot \sum_{i=1}^{m} \Big(F\big(\nu,(q,m,\underbar x_\theta^{(i)})\big)+F\big(\nu,(q,m,\underbar x_{(1-\theta)}^{(i)})\big)-
2F\big(\nu,(q,m,\underbar x)\big)\Big)
\\
&\;
+qr\cdot \sum_{i\neq j \in \{1,...,m\}} \Big(F\big(\nu,(q,m-1,\underbar x_{\theta,1-\theta}^{(i,j)})\big)-F\big(\nu,(q,m-1,\underbar x)\big)\Big)+qm^2r F\big(\nu,(q,m,\underbar x)\big)
\\
&=
\Omega^{(r,\theta),\ast}_{\mbox{\tiny disaster}} F\big(\nu,(q,m,\underbar x)\big)+qm^2r F\big(\nu,(q,m,\underbar x)\big),
\end{aligned}
\end{equation}
which proves (\ref{e:077}).
\smallskip

\noindent {\em Step~2 } We show here that
\begin{equation}
\label{e:078}
\begin{aligned}
\Omega^{(\sigma,K,\lambda)}_{\mbox{\tiny virus branching}} F\big(\nu,(q,m,\underbar x)\big)
&=
\Omega^{(\lambda,K,\sigma),\ast}_{\mbox{\tiny virus branching}} F\big(\nu,(q,m,\underbar x)\big).
\end{aligned}
\end{equation}

To see this, note that
\begin{equation}
\begin{aligned}
\label{e:034}
&\Omega^{\sigma,K,\lambda}_{\mbox{\tiny virus branching}} F\big(\nu,(q,m,\underbar x)\big)
\\
&=
q^{\langle 1,\nu\rangle}\int_{\mathbb{R}_+^m}\sum_{i=1}^{m} \Big(z_i(K-\lambda z_i)\tfrac{\partial}{\partial z_i}f_{\underbar x}\big((z_1,...,z_m)\big)
+ \sigma z_i \tfrac{\partial^2}{\partial {z_i}^2}f_{\underbar x}\big((z_1,...,z_m)\big)\Big)\,\nu^{\otimes m,\downarrow}(\mathrm{d}\underline{z})
\\
&=
q^{\langle 1,\nu\rangle}\int_{\mathbb{R}_+^m}\sum_{i=1}^{m} \Big(-z_i(K-\lambda z_i)x_i f_{\underbar x}\big((z_1,...,z_m)\big)
+ \sigma z_i x_i^2 f_{\underbar x}\big((z_1,...,z_m)\big)\Big)\,\nu^{\otimes m,\downarrow}(\mathrm{d}\underline{z})
\\
&=
q^{\langle 1,\nu\rangle}\int_{\mathbb{R}_+^m}\sum_{i=1}^{m} \Big(
\big(x_i K \frac{\partial}{\partial x_i}+\lambda \frac{\partial}{\partial (x_i)^2}\big)f_{\underbar x}\big((z_1,...,z_m)\big)
- \sigma x_i^2\frac{\partial}{\partial x_i} f_{\underbar x}\big((z_1,...,z_m)\big)\Big)\,\nu^{\otimes m,\downarrow}(\mathrm{d}\underline{z})
\\
&=
\Omega^{(\lambda,K,\sigma),\ast}_{\mbox{\tiny virus branching}} F\big(\nu,(q,m,\underbar x)\big),
\end{aligned}
\end{equation}
which proves (\ref{e:078}).
\end{proof}

We close this subsection by finishing the proof of Theorem~\ref{T:001}.

\begin{proof}[Proof of Theorem~\ref{T:001}]
We have already seen that our family $\{Z^\zeta;\,\zeta>0\}$ is tight and that every limit process satisfies the martingale problem considered in the statement of Theorem~\ref{T:001}. It remains to show {\em uniqueness}. For that we use the duality relation stated in Theorem~\ref{T:003}
together with the fact that the family ${\mathcal K}$ if duality functions is convergence determining (compare with Remark~\ref{Rem:003}).
Moreover, the uniqueness implies the convergence of $Z^\zeta$ as $\zeta\to 0$ to the unique solution of the martingale problem follows.
\end{proof}

\section{Proof on the long term behavior}
\label{S:longterm}
In this section we apply our duality relation to prove Proposition~\ref{P:004}.

\begin{proof}[Proof of Proposition~\ref{P:004}] Consider $\nu\in{\mathcal N}_f(\mathbb{R}_+)$ with $\langle 1,\nu\rangle=1$.
Apply (\ref{ASAW-duality}) with $q=1$ and $m=1$. Then
\begin{equation}
\label{e:079}
   \mathbb{E}_{\nu}\big[\int_{\mathbb{R}_+} e^{-z x}\,Z_t(\mathrm{d}z)\big]
    =
  e^{rt}\int_{\mathbb{R}_+}\mathbb{E}_{x}\big[e^{-z X(t)}\big]\,\nu(\mathrm{d}{z}).
\end{equation}

If we put $\hat{Z}_t:=e^{-rt}Z_t$, then
\begin{equation}
\label{e:079}
   \mathbb{E}_{\nu}\big[\int_{\mathbb{R}_+} e^{-z x}\,\hat{Z}_t(\mathrm{d}z)\big]
    =
  \int_{\mathbb{R}_+}\mathbb{E}_{x}\big[e^{-z X(t)}\big]\,\nu(\mathrm{d}{z})
  \tto 1=\int_{\mathbb{R}_+} e^{- z\cdot x}\,\mathbb{E}\big[\mu(\mathrm{d}z)\big]
\end{equation}
for a random measure $\mu\in{\mathcal M}_f(\mathbb{R}_+)$ with $\mu((0,\infty))=0$, almost surely,  and $\langle 1,\mu\rangle$
is exponentially distributed with mean $1$. Thus the claim follows.
\end{proof}

\bibliography{branching-populations}
\bibliographystyle{alpha}

\end{document}